\documentclass[11pt]{amsart}

\usepackage{amssymb,amsrefs,amscd,skak,xcolor}
\usepackage{tikz}
\usepackage{tikz-cd}
\usetikzlibrary{matrix,arrows}

\usepackage{hyperref}

\usepackage[a4paper,twoside,left=2.8cm,right=2.8cm,top=3.1cm,bottom=2.3cm]{geometry}
\usepackage{lmodern}
\usepackage{multirow}

\BibSpec{article}{%
  +{}{\PrintAuthors}  		{author}
  +{,}{ \textit}     		{title}
  +{,}{ }             		{journal}
  +{}{ \textbf}       		{volume}
  +{}{ \parenthesize} 		{date}
  +{,}{ }      	      		{conference}
  +{,}{ }      	      		{book}
  +{,}{ }            		{pages}
  +{,}{ }            	 	{note}
  +{,}{ }            	 	{status}
  +{, available at}{  \texttt } {eprint}
  +{.}{}              {transition}
}
\BibSpec{book}{%
  +{}{\PrintAuthors}  {author}
  +{,}{ \textit}      {title}
  +{}  { \PrintEditorsB} 		    {editor}
  +{,} { }            {series}
  +{,} { \voltext}    {volume}
  +{,} { }            {publisher}
  +{,}{ }             {date}
  +{,}{ }	      {note}
  +{.}{}              {transition}
}

\BibSpec{collection.article}{%
+{}  {\PrintAuthors}                {author}
+{,} { \textit}                     {title}
+{.} { }                            {part}
+{:} { \textit}                     {subtitle}
+{,} { \PrintContributions}         {contribution}
+{,} { \PrintConference}            {conference}
+{}  {\PrintBook}                   {book}
+{,} { }                            {booktitle}
+{}  { \PrintEditorsB} 		    {editor}
+{,} { }            		    {series}
+{,} { \voltext}    		    {volume}
+{,} { }          		    {publisher}
+{,} { \PrintDateB}                 {date}
+{,} { pp.~}                        {pages}
+{,} { }                            {status}
+{.} {}                             {transition}
}

\newtheorem{theorem}{Theorem}[section]
\newtheorem*{theorem*}{Theorem}

\theoremstyle{plain}
\newtheorem{corollary}[theorem]{Corollary}
\newtheorem{lemma}[theorem]{Lemma}
\newtheorem{proposition}[theorem]{Proposition}

\newtheorem*{lemma*}{Lemma}
\newtheorem*{proposition*}{Proposition}

\theoremstyle{definition}
\newtheorem{definition}[theorem]{Definition}

\newtheorem{remark}[theorem]{Remark}

\newcommand{\CC}{\mathbb{C}}

\newcommand{\RR}{\mathbb{R}}

\newcommand{\ZZ}{\mathbb{Z}}

\newcommand{\fgg}{\mathfrak{g}}
\newcommand{\fll}{\mathfrak{l}}

\newcommand{\foo}{\mathfrak{o}}
\newcommand{\fss}{\mathfrak{s}}

\newcommand{\hathatrho}{\hat{\hat \rho}}

\newcommand{\calA}{\mathcal{A}}
\newcommand{\calB}{\mathcal{B}}
\newcommand{\calC}{\mathcal{C}}

\newcommand{\calF}{\mathcal{F}}
\newcommand{\calH}{\mathcal{H}}

\newcommand{\calL}{\mathcal{O}}
\newcommand{\calM}{\mathcal{P}}
\newcommand{\calN}{\mathcal{N}}

\newcommand{\calR}{\mathcal{R}}

\newcommand{\calV}{\mathcal{V}}
\newcommand{\calLK}{\mathcal L\mathcal K}
\newcommand \restrict[2]{\left. #1\right|_{#2}}
\newcommand \restrictplus[2]{\left. #1\right|^+_{#2}}

\newcommand \barbar[1]{\bar{\bar{#1}}}
\newcommand \hathat[1]{\hat{\hat{#1}}}

\newcommand{\inv}{^{-1}}

\DeclareMathOperator{\rot}{rot}
\DeclareMathOperator{\barrot}{\overline{rot}}
\DeclareMathOperator{\ev}{ev}
\DeclareMathOperator{\Val}{Val}
\DeclareMathOperator{\Curv}{Curv}

\DeclareMathOperator{\glob}{glob}

\DeclareMathOperator{\linspan}{span}

\DeclareMathOperator{\sgn}{sgn}

\DeclareMathOperator{\ang}{Ang}
\DeclareMathOperator{\Null}{Null}

\DeclareMathOperator{\nc}{nc}

\begin{document}
\title{Riemannian curvature measures}
\author{Joseph H.G. Fu and Thomas Wannerer}
\thanks{JHGF was supported by NSF grant DMS-1406252. He wishes also to thank the Courant Institute at NYU, and the Technical University of Vienna, for their kind hospitality during the preparation of this paper.
TW was supported by DFG grant WA 3510/1-1. The authors also wish to thank  A. Bernig and G. Solanes for their many helpful 
comments.}

\begin{abstract} 
A famous theorem of Weyl states that if $M$ is a compact submanifold of euclidean space, then the  volumes of small tubes about $M$ are given by a polynomial in the radius $r$, with coefficients that are expressible as integrals of certain scalar invariants of the curvature tensor of $M$ with respect to the induced metric. It is natural to interpret this phenomenon in terms of  curvature measures and smooth valuations, in the sense of Alesker, canonically associated to the Riemannian structure of $M$. This perspective yields a fundamental new structure in Riemannian geometry, in the form of a certain abstract module over the polynomial algebra $\mathbb R[t]$ that reflects the behavior of Alesker multiplication. This module encodes a key piece of the array of  kinematic formulas of any Riemannian manifold on which a group of isometries acts transitively on the sphere bundle.
We illustrate this principle in precise terms in the case where $M$ is a complex space form.
\end{abstract}

\email{joefu@uga.edu}
\email{thomas.wannerer@uni-jena.de}

\subjclass[2010]{}
\date{\today}
\maketitle

\section{Introduction} 
%

The so-called Weyl tube formula states that if $M \hookrightarrow \RR^N$ is a smooth isometric embedding of a compact 
smooth Riemannian manifold, then the volume of a tube around $M$ of sufficiently small radius $r>0$ is given by a polynomial 
of degree $\le N$ in $r$, whose coefficients may be expressed as integrals of scalar invariants (Lipschitz-Killing curvatures) of the 
curvature tensor of $M$. Up to 
scale, the coefficient in degree $N$ is the Euler characteristic $\chi(M)$. That this invariant admits such an expression is the Chern-Gauss-Bonnet theorem, proved by Chern \cites{chern, chern2}  by integrating certain canonical differential 
forms  on $M$ and its tangent sphere bundle $SM$, derived from  the Cartan apparatus of curvature and connection forms. The coefficients of the lower 
degree terms arise similarly  \cite{cms}. The tube coefficients may be localized by integrating these forms over subsets, and extended to subsets $M$ more general than smooth submanifolds, yielding Federer's theory of curvature measures,  formulated in terms of normal cycles in \cite{zahle}.

From a different perspective, the tube coefficients coincide up to scale with natural extensions of the {\it intrinsic volumes} $\mu_k$ of Hadwiger 
\cite{hadwiger}. These are the 
principal examples of the concept of {\it convex valuation}. Alesker's recent work \cites{ale icm, ale vals 1, ale vals 2, ale vals 3, ale survey, crm, alesker_bernig12} introduces for any smooth manifold $M$ the space $\calV(M)$ of {\it smooth valuations on } $M$, equipped with a natural commutative product. With respect to this product, the intrinsic volumes then appear (up to scale) as powers of $\mu_1$. A smooth immersion $M \hookrightarrow N$ induces a restriction homomorphism $r_{MN}:\calV(N) \to \calV(M)$.  Combining these facts with the Nash embedding theorem, and taking $N$ to be euclidean space, Alesker  observed that Weyl's theorem  yields a canonical embedding  $i_M:\RR[t]/(t^{\dim M + 1}) \hookrightarrow \calV(M)$ associated to any smooth Riemannian $M$, with generator $t$ identified with $\restrict {\mu_1}M$. This system of embeddings of algebras is obviously {\it reproductive} in the sense that if $P \hookrightarrow M$ is again an isometric immersion, then $i_P = r_{PM}\circ i_M$. The image of $i_M$ is the {\it Lipschitz-Killing algebra} $\calLK(M)$.

The fundamental impetus  for the present paper is to understand these phenomena in purely valuation-theoretic terms, and without reference to existence of isometric immersions in euclidean spaces. We accomplish this by introducing a fundamental new structure in Riemannian geometry, a further refinement of the general picture described above. In order to describe it we recall that any smooth valuation  may be localized, albeit non-uniquely.
The resulting space $\calC(M)$ of {\it curvature measures on $M$} carries the structure of a module over the algebra $\calV(M)$. This language is described in Section
 \ref{subsect:curvs n vals} below.
We  define the space $\calR$ of {\it Riemannian curvature measures}, abstracted from the space of summands of the tube coefficients described above. An element $\Psi\in \calR$ is an object that assigns to any smooth Riemannian manifold $M$ a concrete curvature measure $\Psi^M$ on $M$, giving rise to a canonical subspace $\calR(M) \subset \calC(M)$. 
%
Our  main results, given in Section \ref {sect:riem}, may be stated in general terms as follows.
We describe the universal behavior of the elements of $\calR$ under isometric immersion of one Riemannian manifold into another (Theorem \ref{thm:isometric 1}), and show that  the Lipschitz-Killing curvatures are precisely the elements invariant under all such immersions (Theorem \ref{thm:LK isometric invariance}).
This is accomplished  via a natural 
identification \eqref{eq:C code}
of $\calR$ with the space $\RR[[\xi,\eta]]$ of bivariate formal power series. 
We then  give an explicit description of an action of $\RR[t]$  in these terms, reflecting the universal action of $\calLK(M)$ on $\calR(M)$ (Theorem \ref{thm:LK}). 

 The arguments in Section 3 correspond to the soft part of the {\it template method}, a familiar procedure in integral geometry: one proves the existence of a formula of a certain type (typically the hard part), then evaluates the constants by examining enough special cases, or templates (the soft part). Surprisingly, the cases of spheres of varying radius are enough. The technical heart of the  paper is Section \ref{sect:lemmas}, giving the proofs of Lemmas \ref{lem:exist iso} and \ref {lem_module_existence} (the hard part). There we adapt the classical  method of moving frames to display the geometric processes of Alesker multiplication by $\mu_1$, and of isometric immersion, in terms of  formal models based on the Cartan apparatus. We lay the foundations for these models in Section \ref{sect:cartan}.

Key to recent progress in integral geometry is the fact that if $M$ admits a group $G$ of isometries acting transitively on $SM$, then Alesker multiplication on the spaces $\calV^G(M), \calC^G(M)$ of $G$-invariant valuations and curvature measures is intimately related to the array of kinematic formulas for $(M,G)$ (\cite{bfs}, Section 2). Since necessarily $\calR(M) \subset \calC^G(M)$, and $\calLK(M)\subset \calV^G(M)$, the structure studied here represents a universal component of any such array.
 The final Section \ref{sect:hermitian} applies this observation  to the case of complex projective (or hyperbolic) spaces $\CC P^n$. Stabilizing by taking the dimension $n \to \infty$, it turns out that $\calR(\CC P^\infty)$ is a faithful copy of $\calR$, and may be characterized as the space of all invariant curvature measures  that enjoy the fundamental geometric property of {\it angularity}. This space admits a natural basis $\Delta_{kp}$ adapted to the complex structure, distinct from the natural basis for $\calR$. 
 In Theorem \ref{thm:t action} we  translate the formulas of  Theorem \ref{thm:LK} in terms of the $\Delta$ basis, by means of certain simple yet remarkable transforms $\calL, \calM$ on the space of bivariate power series. These transforms already played a role in Section 3.3 of \cite{bfs}, although the relation between that appearance and this one remains mysterious.

\section{Preliminaries}\label{sect:prelims}
%

\subsection {Curvature measures and valuations}\label{subsect:curvs n vals}  A more detailed account of the notions in the present section appears in Section 2.2 of \cite{bfs}.

Although the theory of smooth valuations on a manifold $M$  is entirely independent of  orientation or orientability, it will be convenient to assume that $M$ is oriented. With this assumption we can frame the discussion in terms of integration,  over the normal cycles associated to sufficiently regular subsets of $M$, of smooth differential forms on the cosphere bundle of $M$. 
If the orientation is switched to its opposite, then both the differential forms that we consider and the normal cycles change sign, so that the resulting set function is unchanged. If $M$ is not oriented (even unorientable) it is possible to recast the whole theory  in terms of differential forms twisted by the orientation line bundle of $M$, but, since the theory is essentially local in nature, one may alternatively avoid any loss of generality by working on the orientation double cover of $M$ if needed.

Another simplification available in the present Riemannian context is to replace the cosphere bundle by the sphere bundle $SM$. Let $\dim M = m+1$. In this language, we recall  that any 
closed submanifold with corners (which we refer to henceforth as a {\it simple
smooth polyhedron}) 
 $P \subset M$, admits a {\it normal cycle } $\nc(P)$, a closed integral current of dimension $m$ in the tangent sphere bundle $SM$. The current $\nc(P)$ is {\it Legendrian}, in the sense that it annihilates any multiple of the contact (canonical) 1-form  (and hence also any multiple of its exterior derivative). If $M $ is a submanifold of $N$, we distinguish between the normal cycles with respect to the two different ambient spaces by the notations $\nc_M,\nc_N$.
One may assign to each pair $(\psi, \phi)\in \Omega^m(SM) \oplus \Omega^{m+1}(M)$ the {\bf curvature measure} $\Psi$ that associates to such $P$ the signed measure 
\begin{equation}\label{eq:def curv}
\Psi(P,\cdot):= \pi_* \left(\nc(P) \with \psi \right) +  \int_{\cdot \cap P} \phi.
\end{equation}
We denote this element $\Psi$ by
$[\psi,\phi]$.
The space of all such curvature measures on $M$ is denoted $\calC(M)$.
Such a pair $(\psi,\phi)$ determines also  a {\bf valuation} given by $\mu(P):=  \int_{\nc(P)} \psi  +  \int_{ P} \phi $ for {\it compact } submanifolds with corners $P$. The space of all such set functions is denoted $\calV(M)$, and the assignment $\Psi \mapsto \mu$ is the {\bf globalization map}
$$
\glob:\calC(M) \to \calV(M).
$$

The classical examples are the Federer curvature measures $\Phi_0,\dots, \Phi_{m+1}\in \calC(\RR^{m+1})$ \cites{cm, zahle} and the corresponding intrinsic volumes $\mu_i := \glob \Phi_i$. The domain of geometric shapes $P$ subject to these set functions may be enlarged to the class of sets with positive reach, or still larger classes \cite{fu sand}. However, these extensions are irrelevant to the present paper, in which we emphasize the set functions $\Psi, \mu$, and the $P$ play the role of test objects.

If $\Psi \in \calC(M)$, and $P_0,P_1,P_2,\dots$ are  closed submanifolds with corners such that $\nc(P_i) \to c \nc(P_0)$ in the flat metric topology, then the associated signed  measures on $M$ converge weakly, which we write as
\begin{equation}\label{eq:weak convergence}
\Psi(P_i,\cdot) \rightharpoonup c\Psi(P_0,\cdot)
\end{equation}
It is natural to think of this process in terms of the specialization of a  family of constructible functions \cite{fu-mcc}. In the simplest instance, the $P_i$ are $k$-dimensional spheres of radii $r_i \to 0$ and common center $x$, and $P_0= \{x\}$, in which case $c = 1+ (-1)^k$.

If $f:M \to N$ is a smooth immersion then there are restriction (or pullback) maps $\calC(N) \to \calC(M)$ and $\calV(N) \to \calV(M)$, both of which we denote by $f^*$, given by
$$
(f^*\Psi) (P, E) := \Psi(f(P),f(E)),\quad (f^*\mu)(P): = \mu(f(P)).
$$
Clearly these pullbacks maps commute with globalization.
The point is that the pulled back objects may again be represented by differential forms on the domain manifolds $M, SM$. We carry this out explicitly in the special case of an isometric immersion of Riemannian manifolds in Proposition \ref{prop:iota*} below. We will also use the standard notation
$$
(f_*m)(E):= m(f\inv(E))
$$
for the pushforward of a (signed) measure $m$ by a map $f$. Thus in the circumstances above
$$
f_*((f^*\Psi)(A,\cdot))(E)= \Psi(f(A),E).
$$

The main facts that we use in this paper from the theory of valuations are summarized in the following.

\begin{theorem}[\cites{ale icm, fu 15}]\label{thm:product} \
\begin{enumerate}
\item The space $\calV(M)$ admits a natural commutative multiplication (Alesker product), with the Euler characteristic $\chi$ acting as the multiplicative identity. Furthermore $\calV(M)$ acts on $\calC(M)$ in a natural way, compatible with the product of valuations, i.e. if $\mu \in \calV(M), \Psi \in \calC(M)$ then $\glob(\mu \cdot \Psi) = \mu \cdot\glob(\Psi)$. If $f$ is a smooth immersion as above then $f^*$ is an algebra and module homomorphism, i.e. if $\mu, \nu \in \calV(N), \Psi \in \calC(N)$ then
$$
(f^*\mu)\cdot (f^*\nu) = f^*(\mu \cdot \nu), \quad (f^*\mu)\cdot (f^*\Psi) = f^*(\mu \cdot \Psi). 
$$
\item Suppose $X\subset M$ is a compact simple smooth polyhedron, and $P\times M \to M$ is a smooth proper family of diffeomorphisms $\varphi_p:M\to M, \ p \in P$, equipped with a smooth measure $dp$. Suppose further that the map $P \times S^*M \to S^*M$, induced by the derivative maps $\varphi_{p*}:S^*M \to S^*M$, is a submersion.
 Then $\nu(A):= \int_P \chi( \varphi_p(X) \cap A) \, dp$ defines a smooth valuation $\nu \in \calV(M)$. Given $\mu \in \calV(M), \Psi \in \calC(M)$ we have
\begin{align*}
(\nu \cdot \mu) (A) &= \int_P  \mu(\varphi_p(X) \cap A) \, dp, \\
(\nu \cdot \Psi) (A,E)  &= \int_P  \Psi(\varphi_p(X) \cap A,E) \, dp. \qquad\qquad \square
\end{align*}

\end{enumerate}
\end{theorem}
We will rely on an explicit construction of the Alesker product in terms of the underlying differential forms, originally due to Alesker and Bernig, and stated in Theorem \ref{thm:def mult} below.

\subsubsection{Angular curvature measures}\label{sect:angular} We recall briefly this concept from Section 2.5 of \cite{bfs}. 

If $V$ is an $n$-dimensional real affine space of dimension $n$, we denote by $\Curv(V)\subset \calC(V)$ the space of translation-invariant curvature measures on $V$. Any element of $\Curv(V)$ may be expressed as $[\psi,\phi]$ where both $\psi, \phi$ are translation-invariant. The space $\Curv(V)$ is graded by degree $k\in \{0,\dots ,n\}$.

%

 Let us now suppose further that $V$ is a euclidean space.  A curvature measure $\Psi\in \Curv(V)$ is  {\it angular} if there exists a function $c_\Psi$ on the $k$-Grassmannian of $V$ such that for any convex polytope $P$ 
\begin{equation}\label{eq:def angular}
\Psi (P,\cdot) = \sum_{k=0}^n \sum_{F \in \mathfrak F_k(P)} c_{\Psi} (\vec F) \angle (F,P) \restrict{\mathcal H_k} F
\end{equation}
where $\frak F_k(P)$ is the set of all $k$-faces of $P$, $\angle (F,P)$ is the exterior solid angle of $P$ along $F$, and $\vec F$ is the element of the $k$-Grassmannian parallel to $F$. Any such $\Psi$ is clearly translation-invariant. Any element of $\Curv(V)$ of degree $\ge n-1$ is vacuously angular. 

Passing now to a Riemannian manifold, Section 2.2.2 of \cite{bfs} describes a canonical isomorphism  $\calC(M) 
\leftrightarrow \Gamma(\Curv(TM))$, the space of smooth sections of the  bundle over $M$ whose fiber over $x$ is the space of translation-invariant curvature measures in $T_xM$. Thus $\calC(M)$ inherits a 
natural grading. A curvature measure $\Psi \in \calC(M)$ is said to be  angular iff each value of 
the associated section has this property.

\subsection{The formal  Cartan apparatus}\label{sect:cartan} We give a brief but explicit account of a formal algebraic model for the calculus of moving frames on an oriented Riemannian manifold, implicitly used in classical Riemannian geometry, notably the work of Chern \cites{chern, chern2}. In these and related calculations we adopt the usual summation convention.

\subsubsection
{$\calA_m, \bar \calA_m, \bar \calA_m^+,  \calB_m,  \calB_m^+$, and their realizations}
Let $\calA_m$ denote the anticommutative  graded algebra with generators 
$$
\theta_i, \quad \omega_{ij}= -\omega_{ji}, \quad \Omega_{ij}=-\Omega_{ji}, \quad 0\le i,j\le m,
$$
of respective degrees 1,1, and 2, equipped with the left action of the group $O(m+1)\owns g$ given by
\begin{align}
  L'_g \theta_i & := g_{j i} \theta_{j}\notag\\ 
  L'_g \Omega_{ij} & := g_{k i}g_{l j} \Omega_{kl}\label{O_action}\\ 
  L'_g \omega_{ij} & := g_{k i}g_{l j} \omega_{kl} \equiv g_{0i} g_{lj} \omega_{0l} + g_{ki} g_{0j} \omega_{k0} \mod \langle\omega_{ij}: 0\notin \{i,j\}\rangle.\notag
\end{align}
This algebra admits the formal differential
\begin{align}
   d'\theta_{i} &:= -\omega_{ij}  \theta_j\notag\\
  d'\omega_{ij} & := - \omega_{ik} \omega_{kj}  + \Omega_{ij} \label{structure_equations}\\
    d'\Omega_{ij} & := \Omega_{ik} \omega_{kj} -\omega_{ik} \Omega_{kj},\notag
\end{align}
modeled on the Cartan structure equations, where we follow the conventions of \cite{bishop_crittenden64}, p. 100.
Note that ${d'}^2\ne 0$, since we have not imposed relations corresponding to the Bianchi identity, but this property is unnecessary for our purposes.


Put $\bar \calA_m$ for the quotient obtained by setting all $\omega_{ij}=0$ unless $0\in \{i,j\}$. There is an obvious inclusion map $\bar \calA_m \hookrightarrow \calA_m$. The restriction to $\bar \calA_m$ of the action of $h \in O(m)$ by fixing the $0$ coordinate, denoted $L_h$, descends to $\bar \calA_m$ and intertwines the inclusion and quotient maps. We may define the differential $d$ on $\bar \calA_m$, and also endomorphisms $\tilde L_g, g \in O(m+1)$, by precomposing $d',L'$ with the inclusion and postcomposing with the quotient map. Clearly 
\begin{equation}\label{eq:tilde L}
\tilde L_{gh} =\tilde  L_g\circ L_h, \quad \tilde L_{hg} =  L_h\circ \tilde L_g, \quad d \circ L_h= L_h \circ d
\end{equation}
 for any $ h \in O(m)$.
Put $\bar \calA_m^+\subset \calA_m$ for the subspace of all  $\phi \in \bar \calA_m$ such that $L_g \phi = (\det g) \phi$ for $g \in O(m)$. For $0\le k\le m, 0\le 2p \le k$, we define the elements
\begin{equation}\label{def_hatPhi} 
\bar \calA_m^+\owns\phi_{kp} := \sum_{\pi} \sgn(\pi) \Omega_{\pi_1 \pi_2} \cdots  \Omega_{\pi_{2p-1} \pi_{2p}}  \theta_{\pi_{2p+1}} \cdots  \theta_{\pi_k} \omega_{\pi_{k+1},0} \cdots   \omega_{\pi_{m},0} 
\end{equation}
where the sum extends over all permutations of $\{1,\ldots,m\}$.

Put  $ \calB_m$ for the further quotient of $\bar \calA_m$ obtained by setting all remaining $\omega_{ij}=0$. Now the action of $O(m+1)$ descends.
Put $ \calB_m^+$ for the subspace of all $\psi \in  \calB_m$ such that $L_g \psi = (\det g) \psi, \ g\in O(m+1)$. Define
\begin{equation}\label{def_psi}
 \calB_m^+\owns \psi_{p} := \sum_{\pi} \sgn(\pi) \Omega_{\pi_0 \pi_1}  \cdots \Omega_{\pi_{2p-2}\pi_{2p-1} }  \theta_{\pi_{2p}}  \cdots \theta_{\pi_{m}}
\end{equation}
where the sum extends over all permutations of $\{0,\ldots,m\}$. 


\begin{proposition} \label{prop_inv_elements1} Modulo $\theta_0$, the space of elements of formal degree $m$ in $\bar \calA_m^+$ is spanned by the $\phi_{kp}$. The space of elements of formal degree $m+1$ in $ \calB_m^+$ is spanned by the $\psi_{p}$.
\end{proposition}
\begin{proof} We prove the first assertion. The proof of the second is similar and simpler. Given an $O(m)$ module $M$, we 
refer to the submodule on which the group acts by multiplication by the determinant as the {\it determinant submodule}.

The subspace  of $\bar\calA_m^+$ of elements of degree $m$ decomposes by the degree $p$ in the $\Omega$ and the degree $k-2p$ in the $\theta$. We wish to show that, modulo $\theta_0$, the subspace $\calC_{kp}$ thus described is spanned by $\phi_{kp}$. Put for simplicity $V:= \RR^m$. Let $O(m)$ act on $V$ in the standard way, and on $\fss\foo(m), \fgg\fll(m)$ by conjugation. Clearly $\calC_{kp}/(\theta_0)$ is isomorphic to the $O(m)$-submodule  of the determinant submodule of $\left(\fss\foo(m)^{\otimes p}\otimes V^{\otimes( k-2p)}\otimes V^{\otimes (m-k) }\right)^*$ that consists of elements that are symmetric in the $\fss\foo(m)$ factors and antisymmetric in each of the two clusters of $V$ factors. Pulling back via the projection $\fgg\fll(m) \to \fss\foo(m)$, $M\mapsto (M-M^t)/2$,  in fact
$\calC_{kp}$ is isomorphic to a submodule  of the corresponding submodule of $\left(\fgg\fll(m)^{\otimes p}\otimes V^{\otimes (k-2p )}\otimes V^{\otimes( m-k) }\right)^*\simeq V^{\otimes m *} $.

The First Fundamental Theorem of invariant theory for $SO(m)$ (\cite{procesi07}, Chapter 11.2) states that the algebra of such invariants in variables $y_i\in V$ is generated by functions of the form
 $(y_1,\dots ,y_m)\mapsto \det [y_1\dots y_m]$ and $\langle y_i,y_j\rangle$. In our case the degree is $m$, so if the latter  
appear in any term then all other factors must have the same form. Since these expressions are $O(m)$-invariant, this cannot happen. Thus the determinant submodule is spanned by $y_1\otimes \dots \otimes y_m \mapsto \det [y_1\dots y_m]$. This corresponds  to $\phi_{kp}\in \calC_{kp}$ under the identifications above.
\end{proof}

\subsubsection{Realization maps}
Now let $M$ be a smooth oriented Riemannian manifold of dimension $m+1$, with frame bundle $\hat \pi:FM\to M$ and tangent sphere bundle $\pi: SM \to M$. Thus the total space $FM$ consists of all orthonormal frames $b=(b_0,\dots,b_m)$ for $T_xM$, $x = \hat \pi(b)\in M$.  We may regard $FM$  also as a bundle $z:FM\to SM$ by taking $z(b_0,\dots,b_m):= b_0$. The bundles $\hat \pi, z$ are principal bundles, with groups $O(m+1),O(m)$ respectively. We denote the right action in each case by $R$.

 There is  a well-defined {\it realization map} $\rho :\calA_m \to \Omega^*(FM)$ given by taking the generators of $\calA_m$ to the coframe forms, connection forms, and curvature forms associated to a given $b \in FM$  \cites{bishop_crittenden64, chern}. The  Cartan structure equations state that this map intertwines the formal differential above and  the exterior  differential on $\Omega^*(FM)$. It also intertwines the formal left action $L'$ with the action on $\Omega^*(FM)$ induced by pullback under the right action of the structure group. 
 
 If $P$ is a smooth manifold, with $\bar b: P \to FM$ a smooth map (typically a section of a bundle with total space $FM$),  and $g \in O(m+1)$, then
 \begin{equation}\label{eq:rho covariance}
(R_g\circ\bar b)^* \circ \rho = \bar b^* \circ \rho\circ L'_g. 
 \end{equation}

\begin{definition} \label{def:rho bar} Given  $b=(b_0=\xi_0,b_1,\dots,b_m) \in FM$ we  define $ \bar \rho_b:\bar\calA_m \to \bigwedge^* T_{\xi_0}(SM)$ as follows.
Let $\bar b$ be a local section of the bundle $z$ such that $\bar b(\xi_0) = b$. Then
$\bar \rho_b$ is the restriction to $\bar \calA_m$  of the evaluation at $\xi_0$ of $\bar b^* \circ \rho$.

Define $\bar \rho = \bar \rho_M: \bar \calA_m^+ \to \Omega^*(SM)$ by setting its value at any given $\xi_0\in SM$ by restricting $\bar \rho_b$ to $\bar \calA_m^+$, where  $b=(b_0=\xi_0,b_1,\dots,b_m)$ is any positive adapted frame. 

Abusing notation, we use the same symbol $\bar \rho$ to denote the corresponding maps $\calB, \calB^+ \to \Omega^*(M)$.
 \end{definition}

 \begin{lemma}\label{lem:rho g} The maps $\bar \rho_b, \bar \rho$ are well-defined. If $g \in O(m)$ then $\bar \rho_{R_g b} = \bar \rho_b \circ L_g$.
\end{lemma}
\begin{proof}
We must show that the map $\bar \rho_b$ is   independent of the choice of the extension $\bar b$. This is clearly true of the values  $\bar\rho_b(\theta_i),\bar\rho_b( \Omega_{ij})$. As for $\rho_b(\omega_{0i})$, note first that $\bar \rho_b(\theta_0)$ is the canonical contact 1-form $\alpha$ of $SM$, and hence independent even of the other vectors $b_1,\dots, b_m$ of the frame. Now if $\tilde b_i\in T_{\xi_0}SM$ is the horizontal lift of $b_i \in T_{\pi(\xi_0)}M$ then $\bar \rho_b(\omega_{0i})$ is the interior product of $\tilde b_i$ with $d\alpha$. 

The final conclusion follows from \eqref{eq:rho covariance}, and implies in turn that $\bar \rho$ is well defined.
\end{proof}


\subsubsection{The infinitesimally parallel extension of a frame and the canonical connection of $z$}
\label{sect:infinitesimal parallel}
 It is well known that any frame $b, z(b)= \xi_0 $,  may be extended to a local section $\bar b$ of $z$ such that $\restrict{\bar b^* \circ \rho(\omega_{ij})}{\xi_0} =0 $ if $0 \notin \{i,j\}$. Let us introduce an  explicit terminology for such an extension for later use. Put $x_0: = \pi(\xi_0) \in M$, and let $V\subset M$ be a normal neighborhood of $x_0$. Put $E: V\times S_{x_0} M \to SM$ to be the map induced by parallel translation along geodesics from $x_0$. The {\bf infinitesimally parallel extension} of  $b$ is the adapted moving frame $\bar b$ defined on the neighborhood $\tilde V_{\xi_0}:=E( V\times (S_{x_0} M \setminus \{-\xi_0\}))$  of $\xi_0$, constructed as follows. For $\xi \in S_{x_0}M \setminus \{-\xi_0\}$, define 
$$
\rot^{\xi_0}_{\xi} \in SO(T_{x_0}M) \ {\rm \quad  by \  } \rot^{\xi_0}_{\xi}(\xi_0) = \xi, \ \restrict{\rot^{\xi_0}_{\xi}}{\xi^\perp \cap \xi_0^\perp} = {\rm  identity}.
$$ 
Now set 
$\bar b(\xi) := \rot^{\xi_0}_{\xi} (b):= (\rot^{\xi_0}_{\xi} (b_0),\dots, \rot^{\xi_0}_{\xi} (b_m))$ for $\xi \in S_{x_0} M \setminus \{-\xi_0\}$. In other words, we extend the frame initially to $S_{x_0} \setminus \{-\xi_0\}$ by parallel translation along great circles from $\xi_0$. Now put 
$$
\bar b (E(x, \xi)):= E(x, \bar b(\xi)):= (E(x, \bar b_0(\xi)), \dots , E(x, \bar b_m(\xi))).
$$
The map $b \mapsto \bar b$ is clearly compatible with the right $O(m)$ action on $z:FM \to SM$, i.e.
\begin{equation}
\overline{(R_h b)} = R_h \bar b, \quad h \in O(m).
\end{equation}
Therefore the images of the derivatives $\bar b_*: T_{\xi} SM \to T_b FM$ corresponding to the infinitesimally parallel extensions of all adapted frames $b$ at all points $\xi \in SM$ constitute the family of horizontal subspaces of a {\bf canonical connection } on the principal $O(m)$ bundle $z$.

Consider the standard sphere $S^m \subset \RR^{m+1}$, with standard coordinates $u_0,\dots,u_m$ and standard basis vectors $e_0,\dots,e_m$. For $-e_0 \ne u \in S^m$, put $\barrot_u \in SO(m+1)$ for the element that acts as the identity on vectors perpendicular to $e_0$ and $u$, and such that $\barrot_u(e_0) = u$. Define $\psi_b: S_{x_0}M \to S^m$ by $\psi_b(u_i b_i)_{i=0}^{m} := (u_0,\dots,u_m)$. 
We will need the following obvious facts.

\begin{lemma}\label{lem:rot}  For all $\xi \ne -\xi_0$,
\begin{equation}\label{eq:conjugate rotations}
 \barrot_{\psi_b(\xi)} =  \psi_b \circ \rot^{\xi_0}_{\xi}\circ \psi_b\inv.
 \end{equation}
 If $u\not \perp e_0$, then 
 the map $v\mapsto \barrot_v(u)$
yields a diffeomorphism between small neighborhoods of $e_0, u$.
\end{lemma}

\subsubsection{The virtual antipodal map} Let $a:SM \to SM$ denote the fiberwise antipodal map. Clearly
\begin{equation}\label{eq:a rho}
\bar\rho \circ L_{-1} = a^* \circ  \bar \rho.
\end{equation}

\subsubsection{The immersed version}\label{sect:immersed} We will also need a version of this apparatus that describes in formal terms an isometric immersion of one Riemannian manifold into another. Let $n >m$ be given. The coordinates in the range $0,\dots,m$ will be treated in an essentially different way from the range $m+1,\dots,n$; we denote the former by Roman letters $i,j,\dots$  from the first half of the alphabet and the latter by Greek letters $\alpha,\beta,\dots$. Roman letters $r,s,t,\dots$ from the second half of the alphabet will take values from the entire range $0, 1,\dots, n$.
Put $\calA_{m,n}:= \calA_n/(\theta_\alpha = 0)$. 
 Introduce new symbols 
$$
\bar \Omega_{ij}:= \omega_{i\alpha}\omega_{j\alpha}, \quad  0\le i,j \le m.
$$
Consider the standardly embedded subgroup $O(m+1) \times O(n-m) \subset O(n+1)$,  where the two factors act respectively on the first $m+1$ and last $n-m$ coordinates. The restriction $L_{g,h}$ of the left action of $O(n+1)$ descends to this quotient.

Define the quotient $\bar \calA_{m,n}$ by setting 
\begin{align}
\label{eq:bar mn relations}\omega_{ij} &= 0 {\rm  \quad unless \ } 0\in \{i,j\}, \\
\notag \omega_{\alpha\beta}&= 0 .
\end{align}
There is a natural inclusion into $\calA_{m,n}$. The group $O(m) \times O(n-m)$ acts by fixing the $0$ coordinate. Put 
$$
\bar \calA_{m,n}^{++} \ ({\rm resp.  \ }  \bar \calA_{m,n}^{+}{\rm )} := \{\phi: L_{g,h} \phi = (\det g ) (\det h)\phi  \ ({\rm   resp.  \ } (\det g ) \phi {\rm )}{\rm \ for \ } (g,h) \in O(m) \times O(n-m)\},
$$
 so that the quotient map takes $\bar \calA_n^+ \to \bar \calA^{++}_{m,n}$.

 Put $\barbar \calA_{m,n}$ to be the subspace $\Omega_{\alpha\beta}= \Omega_{\alpha i} =0$. The $O(m) \times O(n-m)$ action descends to this space. Put
  $\barbar \calA^+_{m,n}= \{\phi \in \barbar \calA_{m,n}: L_{g,h}\phi = (\det g) \phi, (g,h) \in O(m) \times O(n-m)\}$.
 The elements of degree $m$
\begin{equation*}\label{eq:good guys boundary}\phi_{kpl} := \sum_{\pi}  \sgn(\pi) \Omega_{\pi_1 \pi_2} \cdots  \Omega_{\pi_{2p-1} \pi_{2p}}\bar \Omega_{\pi_{2p+1} \pi_{2p+2}} \cdots \bar \Omega_{\pi_{2p + 2l-1} \pi_{2p+2l}}  \theta_{\pi_{2p+2l+1} }\cdots  \theta_{\pi_k} \omega_{\pi_{k+1},0} \cdots   \omega_{\pi_{m},0} 
\end{equation*}
belong to this space.

Define the quotient $ \calB_{m,n}$ by setting  all the remaining $\omega_{ij}$
to zero as well. Thus the only remaining formal connection forms are the $\omega_{i\alpha}$.
The group $O(m+1) \times O(n-m)$ acts on this space. Put $ \calB^+_{m,n}$ for the subspace of all elements $\phi$ such that
$L_{(g,h)}\phi = (\det g) \phi$, in particular the elements of degree $m+1$ 
\begin{equation*}\label{eq:good guys interior}\psi_{pl} = \sum_{\pi}  \sgn(\pi) \theta_{\pi_0} \Omega_{\pi_1 \pi_2} \cdots  \Omega_{\pi_{2p-1} \pi_{2p}}\bar \Omega_{\pi_{2p+1} \pi_{2p+2}} \cdots \bar \Omega_{\pi_{2p + 2l-1} \pi_{2p+2l}}  
\theta_{\pi_{2p+2l+1} }\cdots  \theta_{\pi_m} 
\end{equation*}
{are of this type.} 
\begin{proposition} \label{prop_inv_elements2} Modulo $\theta_0$, the space of elements of formal degree $m$ in $\barbar \calA_{m,n}^+$ is spanned by the $\phi_{kpl}$. The space of elements of formal degree $m+1$ in $ \calB_{m,n}^+$ is spanned by the $\psi_{pl}$.
\end{proposition}

\begin{proof} Put $\calN \subset \barbar \calA_{m,n}$ for the subalgebra generated by the $\omega_{i\alpha}$. Clearly 
$\barbar \calA_{m,n} \simeq \bar \calA_m \otimes \calN$, under the convention that  $a\otimes b \;\cdot\; c\otimes d= (-1)^{(\deg b)(\deg c)} ac\otimes bd$. . The groups $SO(m)$ and $O(n-m)$ act on the two factors separately.
Thus any   $\omega\in \barbar\calA_{m,n}^+$  may be expressed in the form
$\omega= \sum \mu\otimes \nu$, where the $\nu\in \calN$ and are $O(n-m)$-invariant.  Abbreviating $W:= \RR^{n-m}$, with coordinates indexed by $\alpha= m+1,\dots,n$, the space 
of all such $\nu$ is isomorphic to the space $S$ of $O(n-m)$-invariant elements of 
$\Lambda^*(W^{m+1})^*$, where the $W$ factors are indexed by $i=0,\dots,m$. 
 The First Fundamental Theorem implies that $S$ is generated by the $\bar \Omega_{ij}= \omega_{i\alpha} \omega_{j\alpha}$.
 
 Thus, as in the proof of Proposition \ref{prop_inv_elements1}, the degree $m$ subspace of $\barbar \calA_{m,n}^+/(\theta_0)$ is isomorphic as an $O(m)$-module to the direct sum over $k,p,l$ of submodules of the determinant submodules of $\left(\fss\foo(m)^{\otimes p}\otimes \fss\foo(m)^{\otimes l}\otimes V^{\otimes k-2p-2l}\otimes V^{\otimes m-k }\right)^*$, where the additional $\fss\foo(m)$ factors correspond to the $\bar \Omega_{ij}$.  The proof now concludes in similar fashion to that of Proposition \ref{prop_inv_elements1}, where the invariant element that emerges now corresponds to $\phi_{kpl}$.
\end{proof}

These algebras may be realized on an  isometric immersion $e:M^{m+1}\hookrightarrow N^{n+1}$. 
 Put $F(M,N) \subset \restrict {FN}M$ for the subbundle of adapted frames $b= (b_0,\dots,b_n)$, i.e. frames such that  $b_0,\dots,b_m\in T_{\hat \pi(b)}M$. This is a principal $O(m+1) \times O(n-m)$ bundle over $M$. 
 The realization map $\rho:\calA_n \to \Omega^*(FN)$ restricts to give an $O(m+1) \times O(n-m)$-equivariant map $\rho_{rel}:\calA_{m,n} \to \Omega^*(F(M,N))$.
The Gauss equation becomes
\begin{equation}\label{eq:gauss}
\rho_{rel}(\bar \Omega_{ij}) =\Omega_{ij}^M-\Omega_{ij}^N,
\end{equation}
 where these two terms refer to the curvature forms of $M,N$ respectively.

Given a fixed frame $b$, we likewise have a map $\bar \rho_{b,rel}:\bar \calA_{m,n} \to \bigwedge{}^*T_{z(b)}(SM)$ since the second fundamental form is well defined.  The restriction of $z$ makes $F(M,N)$ into a principal bundle over $SM$ with group $O(m) \times O(n-m)$.  By restricting to frames $b$ such that $b_0,\dots b_m$ is a positive frame for $T_{\hat \pi(b)}M$, we obtain maps  $\bar \rho_{rel}:\bar\calA_{m,n}^+\to \Omega^*(SM), \calB_{m,n}^+ \to  \Omega^*(M)$ as above. 
Taking $b_{m+1},\dots,b_n$ also to be positive we obtain a map $\bar \calA_{m,n}^{++}\to \Omega^*(SM)$, again denoted $\bar \rho_{rel}$.

\subsection{Fiber integration in a fiber bundle with a connection}\label{sect:fiber} Let $q:P\to X$ be a smooth fiber bundle with model fiber $F$ and group $G$.
Suppose further that $q$ is equipped with  a { $G$-connection}, with horizontal subpaces  $H_p \subset T_pP$. Thus the derivative  $q_*$ induces  isomorphisms $H_p \to T_{q(p)} X$. For $x \in X$, denote by $F_x \subset P$ the fiber over $x$. Given $p \in F_x$, we define the evaluation map 
$$
\ev_p:\Omega^*(P) \to \bigwedge{}^*T_pF_x \otimes \bigwedge{}^*T_x X
$$
by  identifying $T_pP = H_p \oplus T_pF_x \simeq T_xX \oplus T_pF_x $ via the connection. For given $x\in X$, consider the pullback of  the bundle of exterior algebras $\bigwedge{}^*TP$ over $P$ under the inclusion map $F_x \hookrightarrow P$. This gives rise to a {\it restriction map} 
$$
r_x:\Omega^*(P) \to \Omega^*(F_x, \bigwedge{}^*T_x X) 
$$ 
by taking the value  of $\restrict{r_x(\phi)}p$ at $p \in F_x$ to be $\ev_p(\phi)$.

\begin{lemma}\label{lem:fiber}Let $C$ be a $G$-invariant integral current of dimension $k$ in $F$, and for each $x \in X$ let $C_x$ be the corresponding current in the fiber $F_x$. Then the natural fiber integration operator $p_{C*}: \Omega^*(P) \to \Omega^* (X)$ of degree $-k$ is given by
$$
\restrict{(p_{C*} \phi) }x = \int_{C_x} r_x(\phi) \in \bigwedge{}^*T_x X.
$$
\end{lemma}
\begin{proof}
Given $x \in X$, by parallel translation we may find a neighborhood $U\subset X$ of $x$ and a local trivialization $\tau: U \times F \to P$ such that
$D\tau_{x,f}( T_x X \oplus \{0\})= H_{\tau(x,f)}$ for all $f \in F$. Thus, given $\phi \in \Omega^*(P)$ we have $\restrict{\tau^*(\phi) }{x,f} = \ev_{\tau(x,f) }\phi$. The Lemma now follows directly from the definition of  fiber integration.
%
%
%
\end{proof}

\subsection{Generating functions}
We will make extensive use of the Maclaurin expansions 
\begin{equation}\label{eq:maclaurin1}
 (1-x)^{-(k+1)}=\sum_{j=0}^\infty \binom {k + j} j x^j .
 \end{equation}
The special cases $k=\pm\frac 1 2$ admit the alternate expressions
\begin{align}\label{eq:maclaurin1/2}
 (1-x)^{-\frac 1 2 }&=\sum_{j=0}^\infty \binom {2j} j \left(\frac x 4\right)^j ,\\
\label{eq:maclaurin3/2}  (1-x)^{-\frac 3 2 }&=\sum_{j=0}^\infty (2j+1)\binom {2j} j \left(\frac x 4\right)^j .
\end{align}

\section{Riemannian curvature measures} \label{sect:riem}

In this section we  introduce the space $\calR$ of {\it Riemannian curvature measures}, given as the vector space of all possibly infinite linear combinations of certain objects $\{C_{kp}\}_{0\le k,p\in \ZZ,  2p\le k}$. We will describe this situation briefly by saying that $\calR $ is the {\it $\omega$-span} of the $C_{kp}$, and refer similarly to $\omega$-linear maps, etc. 
Formally, an element of  $\calR$ is a universal  object $\Psi$ that associates to every smooth Riemannian manifold $M$ a curvature measure $\Psi^M \in \calC(M)$-- the {\it realization} of $\Psi$ on $M$-- that is expressible in a universal way in terms of the Cartan apparatus. Taking $M$ to be a euclidean space, the Federer curvature measures $\Phi_k $ are included. If $k>\dim M $ then $C_{kp}^M$ is the zero curvature measure. We normalize the $C_{kp}$ so as to minimize the role played by the dimension of $M$. In particular, they are 
 invariant under totally geodesic isometric embeddings (cf. Corollary \ref{cor:totally geodesic} below).
The  {\it Lipschitz-Killing curvature measures} $\Lambda_k$, defined in Section \ref{sect:LK curve} below, are  distinguished elements of $\calR$. We denote their $\omega$-span by $\widetilde{\calLK}$.



\subsection{Definition and basic properties} 

\begin{definition}\label{def R} For $2p\le k$, define $C_{kp}$ to be the object that assigns to any smooth Riemannian manifold $M$  the 
curvature measure
$$\iota_M(C_{kp}):=C^M_{kp}:=\begin{cases} \, \frac{\omega_k}{\pi^k{ (m+1-k)!} \omega_{m+1-k}}[\bar \rho_M( \phi_{kp}), 0] &  \text{if } k< \dim M = m+1 \\
 \frac{\omega_k}{\pi^k}\, [0,\bar \rho_M( \psi_p)]&  \text{if } k= \dim M \\
 0 &  \text{if } k>\dim M.
 \end{cases}
 $$
 Here $\omega_k$ denotes the volume of the unit ball in $\RR^k$. The globalization of $C_{21}^M$ coincides, up to scale, 
 with the valuation first described in Theorem 2 of \cite{bernig voide}.

The $\omega$-span of the  $ C_{kp} $  is denoted by $\calR$.  
The space $\calR$ is $\omega$-graded, by setting $\deg(C_{kp})=k$. Given any smooth Riemannian $M$, the realization map $\iota_M:\calR \to \calC(M)$ is an $\omega$-graded map of degree zero.
We denote the image of the realization map by
$\calR(M):= \iota_M(\calR)\subset \calC(M)$.
\end{definition}


 The normalizing constants in the definition of the $C_{kp}$ are chosen so that
\begin{enumerate}
 \item the $C_{kp}$ are invariant under totally geodesic embeddings (Corollary \ref{cor:totally geodesic});
 \item $C_{kp}^M= \lambda^p C_{k0}^M$ if $M$ is a sphere of curvature $\lambda$;
 \item if $V$ is a euclidean space then $C_{k,0}^V= \frac{k!\omega_k}{\pi^k}\Phi_k$,
where $\Phi_k$ is  the $k$th Federer curvature measure \cite{cm}.
\end{enumerate}
The constants reflect  also the behavior of Alesker multiplication, as will be evident. From this perspective it is convenient to introduce the $\omega$-linear isomorphism 
\begin{equation}\label{eq:C code}
\alpha: \calR \to \RR[[\xi,\eta]], \qquad
  C_{kp } \mapsto \xi^{k-2p}\eta^p.
\end{equation}

\begin{proposition}\label{prop:C is angular} For every $k,p,M$, the curvature measure $C^M_{kp} \in \calC(M)$ is  angular.
\end{proposition}
\begin{proof} Given a point $x \in M$, let $\Psi\in \Curv(T_xM)$ denote the value at $x$ of the section of $\Gamma(\Curv(TM))$ that corresponds to $C^M_{kp}$ under the transfer map described in Section 2.2.2 of \cite{bfs}. Let $P \subset T_xM$ be a convex polytope and $F\subset P$ a face of dimension $k$.
Choosing  moving frames so as to include some fixed orthonormal frame $e_1,e_2,\dots e_k$ for $\vec F$, the expression \eqref{def_hatPhi} implies that the coefficient $c_\Psi(\vec F)$ from \eqref{eq:def angular} is equal to a constant multiple of $\sum_\pi\int_{B_F}  \sgn \pi \,\Omega_{\pi_1\pi_2}\dots\Omega_{\pi_{2p-1}\pi_{2p}} \theta_{\pi_{2p+1}}\dots \theta_{ \pi_k}$, where $\pi$ ranges over all permutations of $1,\dots,k$ and $B_F$ is the unit ball in $\vec F$.
\end{proof}


\subsection{Some facts from spherical integral geometry}\label{sect:some facts}
Our main  calculations all follow ultimately from a few simple yet fundamental observations about the integral geometry of spheres. This discussion also represents the simplest instance of the apparatus just described. We denote  the standard $(n+1)$-dimensional sphere of curvature $\lambda$ (or radius $\lambda^{-\frac 1 2}$) by $S^{n+1}_\lambda$, $\lambda > 0$. 


\begin{proposition}[\cite{fu lag}, 0.4.3]\label{prop:R space forms}Modulo the canonical 1-form and its differential, the space of all differential forms on the sphere bundle of $S^{n+1}_\lambda$ that are invariant under the usual action of $SO(n+2) $ is spanned by the $\bar \rho(\phi_{k,0}), k =0,\dots, n$.
\end{proposition}
\begin{proof} Cf.  pp. 184--5 of \cite{fu lag}.
\end{proof}
Up to scale, these forms correspond to the elements $C_{k,0}^{S^{n+1}} \in \calR(S^{n+1}) \subset \calC(S^{n+1})$, $k\le n$, to which we append the rescaled volume curvature measure $C_{n+1,0}^{S^{n+1}}$.
Since
$\Omega_{ij}=\lambda\, \theta_i\wedge \theta_j $ {for} $ i,j=0,\ldots,n$,
 we have
\begin{equation}\label{eq_Ckp_sphere} 
C^{S^{n+1}_\lambda}_{kp}=\lambda^p C^{S^{n+1}_\lambda}_{k,0}\qquad 
\end{equation}
so that
$$
\calR({S^{n+1}_\lambda}) =\linspan \{C_{k,0}^{S^{n+1}_\lambda}: 0\le k \le n+1\}.
$$
The associated valuations
\begin{equation}\label{eq_def_tau}
\tau_k:=\glob C_{k,0}^{S^{n+1}_\lambda}\in \calV(S^{n+1}_\lambda)
\end{equation} 
constitute a natural basis for the space $\calV^{SO(n+2)}$ of all invariant valuations. Using the identity
$
\pi{(k+1)!}{\omega_k\omega_{k+1}} = {(2\pi)^{k+1}}
$
 we obtain 
 \begin{align}
 \label{eq:cm relation} C_{k0}^{S^{n+1}_\lambda}(S^j_\lambda,\cdot) &= \delta_j^k\frac{k!\omega_k}{\pi^k} \restrict{\calH^j}{S^j_\lambda}\\
\label{eq_Ckp_on_spheres} \tau_k(S^j_\lambda)&=
					 \delta_j^k\ 2 \left(\frac{2}{\sqrt{\lambda}}\right)^k                                      
\end{align}
where $\calH^j$ denotes $j$-dimensional Hausdorff measure.
 Unfortunately this basis is not multiplicative, i.e. in general $\tau_j \cdot \tau_k \ne \tau_{j+k}$, even up to scale. It is helpful to define two other natural bases that {\it are} multiplicative. For the first, 
 let $dg$ denote the Haar probability measure on $SO(n+2)$ and define the valuation
$$
\phi:= \frac{2}{\sqrt\lambda} \int_{SO(n+2)} \chi(\,\cdot\, \cap g S^{n}_\lambda)\, dg
$$
Then by Theorem \ref{thm:product} (2),
\begin{equation}
 \phi^k   = \left(\frac{2}{\sqrt\lambda}\right)^k  \int_{SO(n+2)} \chi(\,\cdot\, \cap g S^{n-k+1}_\lambda)\, dg
\end{equation}
and we again arrive at a basis $\phi^0,\dots,\phi^{n+1}$.

 For the second, 
let
$t= t_{\RR^{n+2}}:=\glob(C_{10}^{\RR^{n+2}})$. By Theorem 2.3.18 of part 2 of \cite{crm}, 
 the Alesker powers of $t$ satisfy
\begin{equation}\label{eq:powers of t}
\frac{\pi^kt^k} {k!} = \omega_k{ \mu_k},
\end{equation}
 where $\mu_k \in \calV(\RR^{n+2})$ is the $k$th intrinsic volume. In particular, $t = \frac 2 \pi \mu_1$.
 We abuse notation and denote again by $t$ the restriction of this valuation to the sphere. Thus $t^0,\dots, t^{n+1}$ constitute another basis for $\calV^{SO(n+2)}$.

We describe the relations among these three bases. Clearly
\begin{equation}
\phi^k (S^j_\lambda)= \begin{cases}
2  \left(\frac{2}{\sqrt{\lambda}}\right)^k  &   j - k \ge 0 \text{  and even}     \\
0 & \text{otherwise}
\end{cases}                          
\end{equation}
 which with \eqref{eq_Ckp_on_spheres} implies that
$$
\phi^k = \sum_{j=0}^\infty \left(\frac \lambda 4\right)^j \tau_{k+2j}.
$$
Thus the generating function  for the coefficients $c_i$ of  $\tau_i$ in the expansion of  $\phi^k$ 
is $ x^k \left(1-\frac {\lambda x^2}{4}\right)^{-1}$,
i.e.
\begin{equation}\label{eq:tau gf}
 \phi^k = \sum_j [x^j]\left(x^k \left(1-\frac {\lambda x^2}{4}\right)^{-1}\right) \tau_j
\end{equation}
where we use the standard generating functions notation
$$
[x^j]\left(\sum_i c_i x^i\right) := c_j.
$$

\begin{lemma}\label{lem_tphi}
 \begin{align}
\label{eq:t0} t^{k} &=  {\phi^k}\left({{1- \frac{\lambda\phi^2}{4}}}\right)^{-\frac k 2} \\ 
\label{eq:t01}&=  \sum_{j=0}^\infty  \binom {{\frac{ k} 2}+j -1}  j  \left(\frac{\lambda}{4}\right)^j \phi^{k+ 2j} \\
\label{eq:t02} &= \sum_{j=0}^\infty  \binom {{\frac{k} 2}+j} j  \left(\frac{\lambda}{4}\right)^j \tau_{k+ 2j}
 \end{align}
\end{lemma}
\begin{proof} 
The relation \eqref{eq:t0} is Proposition 2.4.20 in \cite{crm}. In fact it is easy to prove directly by the template method: simply evaluate both sides on the great spheres $S^j_\lambda$, using \eqref{eq:powers of t} and the known values of the intrinsic volumes of spheres (e.g. \cite{klain rota}) (the case $k=2$ is especially convenient, from which the other cases follow, since $t$ is the unique square root of $t^2$ in $\RR[[\phi]]$ that assigns positive values to curves). The explicit expansion \eqref{eq:t01} then follows from \eqref{eq:maclaurin1}.

To obtain \eqref{eq:t02}, note that by \eqref{eq:tau gf}, \eqref{eq:t0}, 
\begin{equation}\label{eq:tk in terms of tau}
t^k = \sum_j [x^j]\left( {x^k}\left({{1- \frac{\lambda x^2}{4}}}\right)^{-\frac k 2} \left({{1- \frac{\lambda x^2}{4}}}\right)\inv\right)\tau_j.
\end{equation}
\end{proof}

\begin{corollary}\label{cor:t acts on tau}
$$
t\cdot \tau_k = \sum_{j\ge 0} \binom {2j} j \left(\frac \lambda {16}\right)^j \tau_{k+2j+1}.
$$
\end{corollary}
\begin{proof}
By \eqref{eq:tk in terms of tau}, for any formal power series $p$ we have
$$
p(t) = \sum_j [x^j]\left( \left({{1- \frac{\lambda x^2}{4}}}\right)\inv p\left( {x}\left({{1- \frac{\lambda x^2}{4}}}\right)^{-\frac 1 2}\right)\right)\tau_j.
$$
It follows that for any formal power series $q$
$$
t\cdot \sum _j [x^j](q(x)) \tau_j = \sum _j [x^j]\left(x \left(1-\frac {\lambda x^2}4\right)^{-\frac 1 2}q(x)\right) \tau_j .
$$
In particular, this is true of $q(x) = x^k$.
\end{proof}


\subsection{Behavior of the $C_{kp}$ under isometric immersions}\label{sect:behavior isometric} 

\begin{definition} 
 For $2p+2l\le k$, define $C_{kpl}$ to be the object that assigns to any isometric immersion $e:M \hookrightarrow N$ of smooth Riemannian manifolds  the curvature measure
$$ \calC(M) \owns \iota_e(C_{kpl}):=C^{e}_{kpl}:=\begin{cases} \frac{\omega_k}{\pi^k(m+1-k)! \omega_{m+1-k}}\, [ \bar\rho_{rel}(\phi_{kpl}), 0] &  \text{if } k< \dim M = m+1 \\
 \frac{\omega_{k}}{\pi^{k}} [0,\bar\rho_{rel}( \psi_{pl})]  &  \text{if } k= \dim M\\
 0 &  \text{if } k>\dim M.
 \end{cases}
 $$
We denote the $\omega$-span of the $C_{kpl}$, the space of {\it relative Riemannian curvature measures}, by $\calR_{rel}$.
\end{definition}
We introduce the $\omega$-linear isomorphism $ \calR_{rel} \to \RR[[x,y,z]]$ by
\begin{equation}\label{eq: def beta}
\beta: C_{kpj } \mapsto x^{k-2p-2j}y^pz^j.
\end{equation}
The normalizing constants in the definition of the $C_{kpl}$ are chosen such that 
\begin{itemize}
 \item[(1)] $C_{k,0,l}^{e} = C_{k,l}^M$ if $N=\RR^{n+1}$ 
 \item[(2)]  $C_{k,p,l}^{e}= \lambda^p C_{k,0,l}^{e}$ if $N$ is a sphere of curvature $\lambda$.
\end{itemize}

\begin{theorem}\label{thm:isometric 1} Let  $e:M \hookrightarrow N$ be an isometric immersion of one smooth Riemannian manifold into another. Then for all $k,p$
\begin{equation}\label{eq isometric embedding relations}
e^* C^N_{k,p} = \sum_{j=0}^\infty  \binom {\frac{k} 2 +j} j  \left(\frac{1}{4}\right)^j C^{e}_{k+ 2j,p,j}\\
\end{equation}
In other words, if $\delta:\RR[[\xi,\eta]]\to \RR[[ x,y,z]]$ is the $\omega$-linear map 
\begin{equation}\label{eq:isometric gf}
\delta:\xi^k \eta^p \mapsto 
\left(1-\frac{z}{4}\right)\inv 
\left(x\left(1-\frac{z}{4}\right)^{-\frac 1 2}\right)^k
\left(y\left(1-\frac{z}{4}\right)\inv\right)^p
\end{equation}
then $$\iota_e\circ \beta\inv\circ \delta = e^* \circ \iota_N \circ \alpha \inv.$$
\end{theorem}

The proof is based on the following, to be proved in Section \ref{sect:exist iso}.

\begin{lemma}\label{lem:exist iso} There exist universal constants $c^{m,n}_{k,p,j} $ such that
\begin{equation}\label{eq:exist iso}
e^* C^N_{kp}= \sum_j c^{m,n}_{k,p,j} C^{e}_{k+2j,p,j}.
\end{equation}
\end{lemma}

\begin{proof}[Proof of Theorem \ref{thm:isometric 1}] Assuming Lemma \ref{lem:exist iso}, we evaluate the constants in \eqref{eq:exist iso}
 via the template method. 
First, choose $N=S^{n+1}_\lambda$ and let $M=S^{m+1}_\mu$ be a subsphere of curvature $\mu\geq \lambda$. Then 
$$e^* C^N_{kp} = \sum_{j=0}^\infty c^{m,n}_{k,p,j} C_{k+2j,p,j}^{e}= \lambda^p \sum_{j=0}^\infty c^{m,n}_{k,p,j} C_{k+2j,0,j}^{e}. $$
  Since $C^N_{kp}= \lambda^p C^N_{k0}$ we also have 
$$e^* C^N_{kp} = \lambda^p e^* C^N_{k0} = \lambda^p\sum_{j=0}^\infty c^{m,n}_{k,0,j} C^{e}_{k+2j,0,j}.$$
The Gauss equation \eqref{eq:gauss} now yields  $C^{e}_{k+2j,0,j}= (\mu-\lambda)^j C^{e}_{k+2j,0,0}$, so we conclude that $c^{m,n}_{k,p,j} = c^{m,n}_{k,0,j}=:c^{m,n}_{k,j}$. 

In order to determine these constants we consider a standard embedding of $M=S^{m+1}_\lambda$ into $ N=\RR^{n+1}$. Then 
$$
e^* C^N_{k0}= \sum_{j=0}^\infty c^{m,n}_{k,j} C^{e}_{k+2j,0,j} = \sum_{j=0}^\infty c^{m,n}_{k,j} C^M_{k+2j,j}=\sum_{j=0}^\infty c^{m,n}_{k,j} \lambda^j C^M_{k+2j,0} $$
 Globalizing this relation yields
 $$
 t^k =\sum_{j=0}^\infty c^{m,n}_{k,j} \lambda^j \tau_{k+2j}.
 $$
Now \eqref{eq isometric embedding relations} follows from Lemma \ref{lem_tphi}.
\end{proof}

\begin{corollary}\label{cor:totally geodesic} If the embedding $e:M\to N$ is totally geodesic, then $e^*C^N_{kp} = C^M_{kp}$ for all $k,p$.
\end{corollary}

\subsection{Lipschitz-Killing curvature measures}\label{sect:LK curve}
\begin{definition}\label{def_LK} Define the subspace $\widetilde{\calLK}\subset \calR$ 
as the $\omega$-span of the elements
\begin{equation}\label{eq lambda}
\bar \Lambda_{k} := \sum_{j=0}^\infty    \binom{{\frac k 2}+j}{j}  \left(\frac{1}4\right)^{j} C_{k+2j,j}, \quad
k=0,1,\dots.
\end{equation}
\end{definition}
The alternative normalization $ \Lambda_k := \frac{\pi^k}{k!\, \omega_k}\bar \Lambda_{k}$ agrees with the corresponding definition in \cite{bfs}. 
If $k= \dim M$ then $\Lambda_k^M(A,\cdot)$ equals the Riemannian  volume measure restricted to $A \subset M$. However, the $\bar \Lambda_k$ are more natural algebraically, since
\begin{equation}\label{eq:Lambda code}
\alpha( \bar \Lambda_k) = \xi^k\left(1-\frac \eta 4\right)^{-\frac k 2 -1}.
\end{equation}

\begin{theorem}\label{thm:LK isometric invariance} The subspace $\widetilde{\calLK}$ consists precisely of all elements $\Psi \in \calR$ such that 
\begin{equation}\label{eq:iso invariance}
e^*\Psi^N= \Psi^M
\end{equation}
for every smooth isometric immersion $e:M \hookrightarrow N$   of smooth Riemannian manifolds.
\end{theorem}
\begin{proof}
 We introduce an alternate basis for the space of relative Riemannian curvature measures as follows.
Put
$$\Xi_{kpl}=\sum _{\pi} \sgn(\pi) \Omega^N_{\pi_1 \pi_2} \cdots  \Omega^N_{\pi_{2p-1} \pi_{2p}} \Omega^M_{\pi_{2p+1} \pi_{2p+2}} \cdots  \Omega^M_{\pi_{2p + 2l-1} \pi_{2p+2l}}  \theta_{\pi_{2p+2l+1} }\cdots  \theta_{\pi_k} \omega_{\pi_{k+1},0} \cdots 
\omega_{\pi_{m},0} 
$$
for $2p+2l\leq k$, and
$$\Theta_{pl}= \sum _{\pi} \sgn(\pi) \theta_{\pi_0} \Omega^N_{\pi_1 \pi_2} \cdots  \Omega^N_{\pi_{2p-1} \pi_{2p}} \Omega^M_{\pi_{2p+1} \pi_{2p+2}} \cdots  \Omega^M_{\pi_{2p + 2l-1} \pi_{2p+2l}}  \theta_{\pi_{2p+2l+1} }\cdots  \theta_{\pi_m}  
$$
 For $2p+2l\le k$, define $\Gamma_{kpl}\in \calR_{rel}$ to be the object that assigns to any  smooth isometric embedding $e:M^{m+1} \to N^{n+1}$  the curvature measure
$$ \Gamma^{e}_{kpl}:=\begin{cases} \frac{\omega_k}{\pi^k(m+1-k)! \omega_{m+1-k}}\, [ \Xi_{kpl}, 0] &  \text{if } k<m+1 \\
 \frac{\omega_{k}}{\pi^{k}} [0, \Theta_{pl}]  &  \text{if } k=m+1 \\
 0 &  \text{if } k>m+1.
 \end{cases}
 $$
 on $M$. Thus
 \begin{equation}\label{eq:Gamma C}
 \Gamma^e_{k,0,p} = C^M_{k,p}
 \end{equation}
 and
 \begin{equation}
 C_{kpl} = \sum_{j=0}^l (-1)^j \binom l j \Gamma_{k,p+j, l-j}.
 \end{equation}
 In other words, if we define the $\omega$-linear map 
 $$
 \gamma:  \Gamma_{kpj } \mapsto x^{k-2p-2j}y^p z^j
 $$ 
 then $ \gamma\circ \beta\inv(  x^a y^b z^c) =  x^a y^b (z-y)^c$ .

 We  first give a proof, using this formalism, of the well known fact (cf. the Introduction) that $e^*\bar \Lambda^N_k =\bar  \Lambda^M_k$ for any smooth isometric immersion $e:M \hookrightarrow N$. By \eqref {eq:Lambda code} and \eqref{eq:Gamma C} this is encoded in the relation
 \begin{equation}\label{eq:immersion transform}
 e^*\circ\iota_N \circ\alpha\inv\left(\xi^k\left(1-\frac \eta 4\right)^{-\frac k 2 -1}\right)=   \iota_e\circ\gamma \inv\left( x^k \left(1-\frac z{4} \right)^{-\frac k 2 -1}\right).
 \end{equation}
By Theorem \ref{thm:isometric 1}, this relation holds, since
\begin{align*}
  \xi^k\left(1-\frac \eta 4\right)^{-\frac k 2 -1} &\xrightarrow{\ \delta  \ } 
\left(1-\frac{z}{4}\right)\inv  
\left(x\left(1-\frac z{4}\right)^{-\frac 1 2}\right)^k
\left(1-\frac y {4} \left(1-\frac{z}{4}\right)\inv\right)^{-\frac k 2 -1}\\
&=x^k \left(1-\frac z{4} - \frac y 4\right)^{-\frac k 2 -1} \\
&\xrightarrow{\gamma\circ\beta\inv}  x^k \left(1-\frac z{4} \right)^{-\frac k 2 -1}.
 \end{align*}

Conversely, suppose $\Psi = \sum_{kp} a_{kp} C_{kp}\in \calR$ satisfies \eqref{eq:iso invariance} for all isometric immersions $e:M\hookrightarrow N$.
 We wish to show that for each $k\ge 0$ the sequence of coefficients $(a_{k+2j,j})_j $ is proportional to  $\left(\binom{\frac k 2 + j}{j} \left(\frac 1 4\right)^j \right)_j$. We know this to be the case if $\Psi \in \widetilde{\calLK}$. Fixing $k$, it  is therefore enough merely to show that the $a_{k+2j,j}, j >0, $ are determined by $A:= a_{k,0}\frac {\omega_k k!}{\pi^k}$. By induction we may assume that the statement is true for $0,\dots,k-1$; subtracting the  terms that correspond to these values, we may assume that all $a_{jp}=0, \ j\le k-1$.

 Given $j >0$, consider the standard embedding $e: \RR^k \to \RR^{k+2j+1}$. Thus $e^* (\Psi^{\RR^{k+2j+1}})=\Psi^{\RR^k} = A \Phi_k^{\RR^k}$. Taking
$M_\lambda:= \RR^k \times S^{2j}_\lambda$, with standard isometric embeddings $e_\lambda:M_\lambda \hookrightarrow \RR^{k+2j+1}$, as subsets of $\RR^{k+2j+1}$ the normal cycles $\nc(e_\lambda(M_\lambda)) \to \chi(S^{2j}) N(\RR^k)= 2N(\RR^k)$ in the flat metric topology as $\lambda \to \infty$. This is an elementary case of Theorem 3.7 of \cite{fu-mcc}; alternatively it follows from the product formula for normal cycles and the fact that if $rS^{2k}\subset \RR^{2k+1}$ is the standard sphere of radius $r$ then $\lim_{r\downarrow 0}\nc(rS^{2k}) =2 \nc(\{0\})$. Therefore  by \eqref{eq:weak convergence}
$$
\Psi^{\RR^{k+2j+1}}(e_\lambda(M_\lambda),\cdot) \rightharpoonup 2\Psi^{\RR^{k+2j+1}}(\RR^k,\cdot) 
= 2A \Phi_k(\RR^k,\cdot)
$$
as measures on $\RR^{k+2j+1}$. On the other hand, the hypothesis of isometric invariance implies that
$$
\Psi^{\RR^{k+2j+1}}(e_\lambda(M_\lambda),\cdot)  = e_{\lambda *} \left[\Psi^{M_\lambda}(M_\lambda,\cdot)\right]
$$
But it is clear that the signed measures 
$e_{\lambda*}\left[C^{M_\lambda}_{lm}(M_\lambda,\cdot)\right] \rightharpoonup 0 $ for $(l,m) \ne (k+2j,j)$, and that 
$C^{M_\lambda}_{k+2j,j}(M_\lambda,[0,1]^k \times S^{2j}_\lambda) =: c \ne 0$,  independent of $\lambda$.
It follows that  $a_{k+2j,j} = \frac {2A}c$.
\end{proof}

\subsection{Lipschitz-Killing valuations}

The Lipschitz-Killing curvature measures are distinguished also by the fact that their globalizations behave in a universal fashion with respect to Alesker multiplication.
%
%

\subsubsection{The universal module}
Given a smooth Riemannian manifold $M$, we set $\calLK(M)\subset \calV(M)$ for the {\bf Lipschitz-Killing algebra } of $M$, generated by 
\begin{equation}\label{eq:def tM}
t_M := \glob(\bar \Lambda_1^M) \in \calV(M).
\end{equation}

\begin{lemma}\label{lem_module_existence}
Given $k,p,m$, there exists $\Psi_{k,p,m} \in \calR$ such that, if $M$ is any smooth oriented Riemannian manifold of dimension $m+1$, then
$$
t_M \cdot C_{kp}^M =  \Psi_{k,p,m}^M.
$$ 
In particular, $t_M\cdot \calR(M)\subset \calR(M)$ for  every oriented Riemannian manifold $M$.
\end{lemma}
We postpone the proof of Lemma \ref{lem_module_existence} to Section \ref{Section_existence_of_module}
 below.


Assuming  Lemma \ref{lem_module_existence}, we now determine the $\Psi_{kpm}\in \calR$. 
Using the identification \eqref{eq:C code} of $\calR$ with $\RR[[\xi,\eta]]$, we define an  action of $\RR[t]$ on $\calR$ by
\begin{equation}\label{eq:t action 1}
\alpha(t\cdot \Psi ):= \frac{\xi}{\sqrt{1-\frac {\eta} 4}}\alpha(\Psi).
\end{equation}
Thus by \eqref{eq:Lambda code} 
\begin{equation}
t\cdot\bar \Lambda_k =\bar \Lambda_{k+1}.
\end{equation}
Writing out \eqref{eq:t action 1} explicitly:
\begin{align}\label{eq:mult t 1}
t\cdot C_{kp}   &= \sum_{j=0}^\infty \binom{2j} j \frac 1 {16^j} C_{k+2j +1, p+j}, \\
\label{eq:mult t^i} t^{i} \cdot C_{kp} &= \sum_{j=0}^\infty \binom{\frac i 2 +j-1}{j} \frac{1}{4^j}\, C_{k+i+ 2j,p+j} .
\end{align}

\begin{theorem}\label{thm:LK}
The Alesker powers of $t_M$ are given by
\begin{equation}\label{eq:tMk}
t_M^k = \glob(\bar \Lambda_k^M).
\end{equation}
 The algebra  homomorphism $\RR[t] \to \calLK(M)\subset \calV(M)$ determined by  $t \mapsto t_M$ intertwines the actions of $\RR[t]$ on $\calR $ and of $\calLK(M)$ on $\calR(M)$, i.e.
\begin{equation}
(t\cdot \Phi)^M =  t_M \cdot \Phi^M.
\end{equation}
\end{theorem}

\begin{proof}

We must show that 
\begin{equation}\label{eq:t action}
\Psi_{k,p,m} = \sum_{j=0}^\infty \binom{2j}j\left( \frac1{16}\right)^j C_{k+2j+1,p+j}.
\end{equation}

This formula is   determined by  the cases $M= S^{n+1}_\lambda$ for varying $n,\lambda$, as follows.
By Theorem \ref{thm:LK isometric invariance}, the valuation $t\in \calV(S^{n+1}_\lambda)$, described in Section \ref{sect:some facts}, is identical to  $t_{S^{n+1}_\lambda}$ as defined in \eqref{eq:def tM}.
Globalizing $t\cdot C_{kp}$ in $S^n_\lambda$ and using \eqref{eq_Ckp_sphere}, \eqref{eq_def_tau}, and  Corollary~\ref{cor:t acts on tau}, we obtain
$$
\glob(t\cdot C_{kp}^{M})= \lambda^p\, t\cdot \tau_k = \lambda^p \sum_{j=0}^\infty \binom{2j}{j} \left(\frac{\lambda}{16}\right)^j \tau_{k+2j+1}=  \glob\left(\sum_{j=0}^\infty \binom{2j}j\left( \frac1{16}\right)^j C^M_{k+2j+1,p+j}
\right)$$
 as claimed. 
 \end{proof}

\subsubsection{There are no further universal Riemannian valuations}\label{sect:no extras}
Our next result expresses the fact that $\widetilde{\calLK}$
comprises {\it all}  elements of $\calR$
whose globalizations act on $\calR$ in a universal way. 

\begin{theorem} \label{thm:no more} Let $\Phi \in \calR$, and suppose that there exists a map
$L= L_\Phi:\calR \to \calR$ such that for every Riemannian manifold $M$ and every $\Psi \in \calR$
$$
\glob(\Phi^M) \cdot \Psi^M = L(\Psi)^M.
$$
Then $\Phi \in \widetilde{\calLK}$, and $L$ is described by \eqref{eq:t action 1}.
\end{theorem}
 In order to prove this we examine the particular cases $M=\CC P^n_\lambda$, the complex space form of constant holomorphic sectional curvature $4\lambda$, and exploit the detailed understanding of the integral geometry of such spaces achieved in \cites{ bfs, bfs2}. 
The needed  facts
 from hermitian integral geometry are given in the next Proposition. For simplicity we denote the realizations $C_{kp}^{\CC P^n_\lambda}\in \calR(\CC P^n_\lambda)$  simply by $C^\lambda_{kp}$. Put $\calV^n_\lambda$ for the algebra of isometry-invariant valuations on $\CC P^n_\lambda$. Following \cite{bfs}, we denote  the canonical generator of the Lipschitz-Killing algebra by $t_\lambda:= t_{ \CC P^n_\lambda} \in \calV^n_\lambda$, and take $s\in \calV^n_\lambda$  to be the valuation 
 $$
 s(Q) := \int_{\overline{Gr}_{n-1}} \chi (Q \cap P)\, dP
 $$
 where $\overline{Gr}_{n-1}$ is the homogeneous space of all totally geodesic subspaces of $\CC P^n_\lambda$ of complex codimension 1, and $dP$ is the (appropriately normalized) Haar measure (cf. \cite{bfs}, (3.1)).

\begin{proposition} \label{prop:hermitian lemma}
{} \
\begin{enumerate} 
\item\label{item:val = glob ang} The restriction of the globalization map to the space of invariant angular curvature measures gives a linear isomorphism with $\calV^n_\lambda$.
 \item\label {lem_CP_lin_indep} If  $\lambda\neq 0$ and $0\le k\leq n$, then
the curvature measures
$$C^\lambda_{k,0}, C^\lambda_{k,1},\ldots, C^\lambda_{k,\lfloor k/2\rfloor}$$
are linearly independent.
\item\label{item:span} $\calR(\CC P^n_\lambda)$ coincides with the space of all invariant angular curvature measures on $\CC P^n_\lambda$.
 \item \label{item:angularity}  
 Let $p(\tau,\sigma)$ be a polynomial in the formal variables $\tau,\sigma$, where $\deg \tau = 1, \deg \sigma = 2$. 
 If multiplication by $p(t_\lambda,s) \in \calV^n_\lambda$ stabilizes the space of invariant angular curvature measures on $\CC P^n_\lambda$ then
 the terms of $p(\tau,\sigma)$ of weighted degree less than $n-2$ are independent of $\sigma$.
 \end{enumerate}
\end{proposition}


\begin{proof} 
Conclusion \eqref{item:val = glob ang} follows from Lemma 3.9 of \cite{bfs}.
Corollary \ref{cor:surjective}, proved below, yields  \eqref{lem_CP_lin_indep}. Therefore the span of the $C_{kp}$ coincides with that of the $\Delta_{kp}$ (defined in Section \ref{sect:c space forms} below) for $k$ lying in this range. 
Examining the standard embedding of  $\CC P^n_\lambda$ into $\CC P^N_\lambda, N>>n$, conclusion \eqref{item:span} now follows from Lemma 3.6 of \cite{bfs} and Corollary \ref{cor:totally geodesic} above.

To prove \eqref{item:angularity}, recall that the algebra $\Val^{U(n)}$ is the quotient $\RR[s,t]/(f_{n+1},f_{n+2})$, where each $f_i(s,t)$ is a certain polynomial of weighted degree $i$, where $\deg s=2, \deg t =1$. We apply Theorem 2 of \cite{bfs 2}, which states that the algebra of isometry invariant valuations on $\CC P^n_\lambda$ whose action preserves 
the space of angular invariant curvature measures is given by all $p(t_\lambda,s)$ such that 
\begin{equation}\label{eq:test}
ts(4s-t^2)\; \frac{\partial p}{\partial \sigma} \left( \frac{t}{\sqrt{1-\lambda s}} ,s \right) =0 \qquad \text{in}\ \Val^{U(n)}.
\end{equation}

Clearly we may assume that $p(\tau,\sigma)$ is divisible by $\sigma$. 
Let $\tau^i\sigma^j$ be the term of smallest degree of $p(\tau,\sigma)$ (we may assume that the coefficient is unity). 
It follows that  the term of lowest degree of the left hand expression in \eqref{eq:test} is $ts(4s-t^2)t^{i}s^{j-1}$. 
The ideal of relations between $s,t$ in $\Val^{U(n)}$ is generated by homogeneous polynomials in degrees $n+1,n+2$, 
 so the presence of this last monomial implies that $i+2j \ge n-2$.
\end{proof}

\begin{proof}[Proof of Theorem \ref{thm:no more}] Suppose $\Phi\in \calR \setminus \widetilde{\calLK}$ . Take $N$ large enough that this is still true modulo the $\omega$-span of the $C_{kp}, k> N$. By Proposition
\ref{prop:hermitian lemma}, \eqref{item:val = glob ang} and
\eqref{lem_CP_lin_indep},
for sufficiently large $n\ge N$ the globalization $\glob(\Phi^{\CC P^n_\lambda})\notin \calLK(\CC P^n_\lambda)$. Choose a polynomial $p(\tau,\sigma)$ such that 
$p(t_\lambda, s)=\glob(\Phi^{\CC P^n_\lambda})$ in $\calV_\lambda^n$.  By  Proposition~\ref{prop:hermitian lemma}, 
items \eqref{item:angularity} and \eqref{lem_CP_lin_indep},  it follows that $\Phi\in \widetilde{\calLK}$
modulo the $\omega$-span of the $C_{kp}$,  $k\geq n-2$. This is a contradiction for $n\geq N$ sufficiently large.
\end{proof}

\begin{remark} With some loss of precision, conclusions \ref {lem_CP_lin_indep} and \ref{item:span}  above may be paraphrased as: if $\lambda \ne 0$ then the realization map from $\calR$  to $\calC(\CC P^\infty_\lambda)$ gives a linear isomorphism with the invariant angular elements. We will exploit this situation further in Section \ref{sect:correspondence} below.
\end{remark}


\section{Proofs of Lemmas \ref{lem:exist iso} and \ref{lem_module_existence}}\label{sect:lemmas}

\subsection{Proof of Lemma \ref{lem_module_existence}} \label{Section_existence_of_module}

\subsubsection{The Alesker-Bernig formula}\label{sect:ab formula}
 We use the description provided in \cites{alesker_bernig12,fu15} of the Alesker product in terms of differential forms. Given a smooth oriented Riemannian manifold $M^{m+1}$, denote by $\Sigma$ the oriented $(3m+2)$-dimensional manifold 
\begin{equation}\label{eq:def sigma}
\Sigma=  \left\{ (\xi,\zeta,\eta)\in SM\times_M SM\times_M SM\colon  \ \xi\neq \pm \eta,\ \zeta\in \overline{\xi,\eta}\right\}
\end{equation}
where $\overline{\xi,\eta}$ denotes the open geodesic segment in $S_{\pi\xi} M$ joining $\xi,\eta$. We have three associated projections
$$\xi,\zeta,\eta \colon SM\times_M SM\times_M SM \to SM$$
to the first, second and third factors respectively. 

We regard the projection $\xi:SM\times_M SM\times_M SM\to SM$ as a bundle  associated to the principal $O(m)$ bundle $z:FM \to SM$. Here  $O(m)$ acts diagonally on the model fiber $S^m \times S^m \subset \RR^{m+1}\times \RR^{m+1}$ of $\xi$ via the following action on $\RR^{m+1}$: let $u_0,\dots,u_m$ be standard coordinates for $\RR^{m+1}$, and let $O(m)$ fix the $u_0$ coordinate and act as usual on $(u_1,\dots,u_m)$. Thus the fibers $\Sigma_\xi$ correspond to the $O(m)$-invariant subset 
$$
C:=\left\{ { (v,u)}: u\neq \pm e_0,\ v\in \overline{e_0,u}\right\}\subset S^m\times S^m.
$$
We endow the bundle $\xi$ with the  connection induced by the canonical connection on the principal bundle $z$, described in Section \ref{sect:infinitesimal parallel}.

\begin{theorem}[\cites{alesker_bernig12, fu15}]\label{thm:def mult}
 If $\mu \in \calV^\infty(M)$, $\beta\in \Omega^{m}(SM)$, and $\gamma\in\Omega^{m+1}(M)$, then 
 $$\mu\cdot [\beta, \gamma]= [\theta, \psi]$$
with
\begin{equation}\label{eq_ab1}\theta=(-1)^{m+1}  \xi_*  ( \zeta^*\beta \wedge \eta^* a^* \Delta_\mu) + \pi^* \calF_\mu \cdot \beta,\end{equation}
\begin{equation}\label{eq_ab2} \psi = \pi_*(\beta \wedge a^* \Delta_\mu) + \calF_\mu \cdot \gamma.\end{equation} 
Here
$$\calF\colon \calV^\infty(M)\to C^\infty(M), \qquad \Delta: \calV^\infty(M)\to \Omega^{m+1}(SM)$$
are defined by $\calF_\mu(x)=\mu(\{x\})$ and $\Delta_\mu= D\rho + \pi^*\tau$, provided $\mu=[[\rho, \tau]]$. 
\end{theorem}

The sign in front of the first term of \eqref{eq_ab1} corrects the corresponding formula in \cite{fu15}. Here the operator $D\rho$ is the {\it Rumin differential} of $\rho$ \cites{rumin, bb}, defined as the unique $(m+1)$-form $D \rho = d(\rho + \theta_0 \wedge \delta)$ such that
$$
\theta_0 \wedge d(\rho + \theta_0 \wedge \delta)  = 0
$$
for some $(m-1)$-form $\delta$, where $\theta_0 \in \Omega^1(SM)$ is the contact form. Recall that $a\colon SM\to SM$, $a(\xi)=-\xi$ denotes the fiberwise antipodal map.
The fibers {$\Sigma_\xi$} may be compactified so as to make clear the convergence of the fiber integrals. The orientations of the fibers in these integrals are canonical; we need not make them precise here since we only wish to confirm the existence of the universal product formula.
Actually the term $\pi^*\tau$ may be ignored, for trivial dimensional reasons: the integration is of forms of degree $2m+1$ over fibers of dimension $m+1$, so that any terms including a horizontal factor of degree $m+1$ must yield zero. By the same token $\Delta_\mu$ may be perturbed by any horizontal form of degree $m+1$ without affecting the result.

Define
\begin{align}\label{LK_form} 
\bar\calA^+_m \owns \lambda_k&:= \sum_{j=0}^{\lfloor\frac{m-k}{2}\rfloor}   \binom{ \frac k 2 +j}{j}  \left(\frac{1}4\right)^{j}  \frac{\omega_{k+2j}}{\pi^{k+2j}(m-k-2j+1)! \omega_{m-k-2j+1}}\phi_{k+2j,j}  \\
\notag
 &= { \frac{\, \omega_k}{\pi^k \, (m-k+1)!!\, \omega_{m-k+1}} \sum_{j=0}^{\lfloor\frac{m-k}{2}\rfloor}    \frac{1}{2^j j! (m-k-2j)!! }\phi_{k+2j,j}}. \notag
\end{align}
Thus $\bar \rho(\lambda_k)$ is the boundary term of $\bar \Lambda_k \in \calR$, as in Definition \ref{def R}.
Here we use the relation
$$
k\omega_k = 2\pi\omega_{k-2}, \quad k=1,2,\ldots.
$$
For nonnegative integers $k,p$ satisfying $2p\leq k\leq m-1$ we define elements of  degree $m+1$ by
$$
\bar\calA^+_m \owns D_{kp} := (-1)^k \sum_{\pi} \sgn(\pi) \Omega_{\pi_1 \pi_2} \cdots  \Omega_{\pi_{2p-1} \pi_{2p}}  \theta_{\pi_{2p+1}} \cdots  \theta_{\pi_k} \Omega_{\pi_{k+1},0}\omega_{\pi_{k+2},0} \cdots   \omega_{\pi_{m},0} .
$$
The sum ranges over all permutations of $1,\ldots,m$.

The proof of Lemma~\ref{lem_module_existence} using the Alesker-Bernig formula is based on the following observation. 

\begin{lemma}\label{lem:formal d lambda 1} The  formal exterior derivative $d\lambda_1$ is the sum of a multiple of  $\theta_0$ and an element of $\calB_m$.
\end{lemma}
\begin{proof} More precisely and generally, we compute
\label{eq:variation Lambda}
 \begin{align} d\lambda_k = &\frac{ \omega_k}{\pi^k \, (m-k-1)!!\, \omega_{m-k+1}} \times\\ 
\notag &\times  \left[\theta_0 \wedge\sum_{j=0}^{\lfloor\frac{m-k}{2}\rfloor }    \frac{k}{2^j j! (m-k-2j)!! }\phi_{k+2j-1,j}  \notag
+ \delta_{m-1}\right]
\end{align}
where 
$$\delta_{m-1}:= \frac{1}{2^ \frac{m-k-1}{2} \left( \frac{m-k-1}{2}\right)!} D_{m-1, \frac{m-k-1}{2}}$$ 
for $m-k$ odd and $\delta_{m-1}:=0$ otherwise. This follows by telescoping from the formula
\begin{equation*}\label{dPhi_horizontal}  d\phi_{kp} { -} (k-2p) \theta_0\wedge\phi_{k-1,p}   + 2p D_{k-2,p-1} - (m-k)D_{kp} = 0
\end{equation*}
since $\omega_{ij} = 0$ whenever both $i,j \ge 0$, which in turn follows from the structure equations \eqref{structure_equations}.
\end{proof}

\begin{proposition}\label{prop:cartan D} For any smooth oriented manifold $M$ of dimension $m+1$, we have $D(\bar \rho(\lambda_k)) \equiv \bar \rho (d\lambda_k)$ modulo the space of horizontal forms.
\end{proposition}
\begin{proof}
{ Recall that $\bar \rho \circ d = d\circ \bar \rho$.}
 Since $\delta_{m-1}$ has degree $m+1$ and no $\omega_{\cdot 0}$ factors, it follows that  for any Riemannian $M^{m+1}$ 
 the  realization of $d\lambda_k$  is a functional multiple of $\bar \rho(\theta_0)$. 
\end{proof}

 \subsubsection{The formal model}  Since normal cycles are Legendrian, and the realization of $\theta_0$ is the contact form of any $SM$, with Proposition \ref{prop_inv_elements1} the following implies Lemma \ref{lem_module_existence}. 
 \begin{proposition} \label{prop:formal forms} Given $k,p,m$, there exist $\varphi_1 \in { \bar\calA^+_m}, \varphi_2 \in \calB^+_m$ with the following property. Given an oriented Riemannian manifold $M^{m+1}$, let $\theta, \psi$ be as in Theorem \ref{thm:def mult} for $\beta = \bar \rho(\phi_{kp}),  \mu = t_M$, and $\gamma =0$. Then
 \begin{equation}\label{eq:ab terms}
 \theta = \bar\rho(\varphi_1), \quad \psi= \bar \rho(\varphi_2).
 \end{equation}
 \end{proposition}
 
 We will only give explicitly the construction of $\varphi_1$, corresponding to the term of the Alesker-Bernig formula that involves fiber integration over $\Sigma$;  that of $\varphi_2$ is similar but significantly simpler. 

 The gist of the construction is contained in Lemmas \ref{lem:I} and \ref{lem:canonical maps}. The idea is to carry out the fiber integration in the first term of the right hand side of  \eqref {eq_ab1} in purely symbolic terms (the second term vanishes since $\calF_\mu =0$, which is to say that $t_M$ clearly vanishes on singletons).  The underlying observation is  Lemma \ref{lem:formal d lambda 1}, which  implies that in the present setting the forms $\beta, \Delta_\mu$ from 
 \eqref{eq_ab1} are realizations of  elements of  $\bar \calA_m^+$. Lemma \ref{lem:canonical maps} shows that the 
 process of pulling them back under $
 \zeta, \eta$ may  be viewed as the realization of a parallel symbolic process. With the aid of Lemma \ref{lem:fiber}, Lemma 
 \ref{lem:I} shows that the fiber integration step in the Alesker-Bernig formula may also be carried out symbolically, and that 
 the symbolic result belongs to $\bar \calA_m^+$.  Proposition \ref{prop_inv_elements1} then concludes the proof.

%
%
%

 Recalling the left actions of $O(m)$ on $S^m$ and on $\bar \calA_m$, we define the left  action $\Lambda$ of $O(m)$ on $  \Omega^*(S^m,  \bar\calA_m) $ by
 \begin{equation}
 \Lambda_g \phi := L_g \circ ({g\inv}^*\phi )= {g\inv}^*(L_g \circ \phi).
 \end{equation}
 Define similarly an action  $ \hat \Lambda$ of the same group on
 $  \Omega^*(S^m\times S^m,  \bar\calA_m) $ via the diagonal action  on $S^m\times S^m$.
  Let $\tilde \zeta, \tilde \eta: S^m \times S^m \to S^m$ denote the projections to the respective factors; these maps clearly  intertwine  the $O(m)$ actions. Denote by $\Omega^*(S^m,\bar \calA_m)^+$ the subspace on which $O(m)$ acts by multiplication by the determinant, and note that, if $\alpha,\omega \in \Omega^*(S^m,\bar\calA_m)^+$, then  $\tilde\eta^*\alpha \wedge\tilde\zeta^*\omega \in \Omega^*(S^m\times S^m,\bar\calA_m)^{O(m)}$.

 Given $\xi_0\in SM$ with $x_0 := \pi(\xi_0) \in M$, let $b$ be an adapted frame at $\xi_0$. Define the diffeomorphism $ \psi_{ b}:   S_{x_0}M\to S^m $, $ \psi\inv_{ b} ( u):=  u_i  b_i$, and define $\hat \rho_b: \Omega^*(S^m,\bar\calA_m) \to \Omega^*(S_{x_0}M, \bigwedge{}^*T_{\xi_0} SM)$ by
 \begin{equation}\label{eq:def hat rho b}
\hat \rho_b(\phi):= \psi_b^*(\bar \rho_b \circ \phi) = \bar \rho_b \circ (\psi_b^*\phi).
\end{equation}
 If $b, b'=R_g b$ are  adapted frames at $\xi_0\in SM$, where $g \in O(m)$, then by Lemma \ref{lem:rho g}
 $$
 \hat \rho_{b'}\phi= \bar \rho_{R_gb}\circ \psi_{R_g b}^*\phi=\bar \rho_b \circ L_g \circ \psi_b^*((g\inv)^* \phi)=  { \bar \rho_b} \circ  \psi_b^*(L_g \circ(g\inv)^* \phi)=\hat \rho_b (\Lambda_g \phi).
 $$
It follows that this value is independent of the choice of positively oriented frame $b$ if $\phi $ is $SO(m)$-invariant. 

Consider the bundle $SM\times_M SM \to SM$ under projection to the first factor as an $O(m)$ bundle associated to $z$, 
endowed with the connection induced by the canonical connection on $z$. Here  the action of $O(m)$  on the model fiber 
$S^m\subset \RR^{m+1}$ is again induced by the standard action on the $1,\dots,m$ coordinates, and fixes the 0 
coordinate; thus the projections from the bundle $\xi$, obtained by omitting the second or third coordinate, are both bundle 
maps that intertwine the connections.
We then obtain a  map $\hat \rho:\Omega^*(S^m, \bar \calA_m)^{+} \to \Omega^*(SM \times_M SM)$ as the restriction of $\hat \rho_b$ for any choice of positive adapted frame $b$, via the constructions of Section \ref{sect:fiber}.
In the same manner we have $\hathat \rho:\Omega^*(S^m\times S^m, \bar \calA_m)^{O(m)} \to \Omega^*(SM \times_M SM\times_M SM)$.  Factor the maps $\zeta= \pi_2 \circ \zeta_1, \eta=\pi_2 \circ \eta_1$, where the $\zeta_1,\eta_1: SM \times_M SM\times_M SM \to SM \times_M SM$ respectively omit the third and second coordinates. 

 \begin{lemma} \label{lem:I} Let $I$ denote integration over the model fiber
 $C := \{(z,y): y \ne \pm e_0, \ z \in \overline{e_0,y} \}\subset S^m\times S^m$, with an appropriate choice of orientation. Then the following diagram commutes for all $(m+1)$-dimensional oriented Riemannian manifolds $M$:
\begin{equation}\label{cd0}
\begin{CD}
(\Omega^*(S^m,\bar \calA_{m})^+)^{\otimes 2 } & @>\hat \rho^{\otimes 2 }>>& \Omega^*(SM\times_M SM)^{\otimes 2 }\\
@V{\tilde \zeta^*\wedge\tilde \eta^*}VV& & @ V{ \zeta_1^*\wedge \eta_1^*}VV \\
\Omega^*(S^m\times S^m, \bar\calA_{m})^{O(m)} &@>\hathatrho>>  &  \Omega^*(SM\times_M SM\times_M SM)\\
@VIVV& & @ V{\xi^\Sigma_*}VV \\
\bar\calA^+_m & @>\bar \rho >> &  \Omega^*(SM) 
\end{CD}
\end{equation}
Here $\xi^\Sigma$ denotes the restriction of the projection $\xi$ to the submanifold $\Sigma$ of \eqref{eq:def sigma}.
\end{lemma}

\begin{proof} 
That the top square commutes follows  straightforwardly from the definitions.

 Lemma \ref{lem:fiber} implies that a variant of the  bottom square commutes, in which  the invariance conditions on the spaces on the left are relaxed, 
 and the maps ${ \bar \rho}, \hathat \rho$ are replaced by ${\bar \rho_b}, \hathat \rho_b$ for some frame $b$ at $\xi_0$.
Since the orientation of $C$ is reversed under the action of $g \in O(m)$ if $\det g = -1$, we find that
 $$
 I(\hat\Lambda_g \phi) = ( \det g) \,  L_g I(\phi),  \quad g\in O(m).
 $$
Thus the restriction of $I$ to the subspace of invariant elements indeed takes values in $\bar\calA_m^+$, and the given diagram commutes.
 \end{proof}

 The next step is to use $\hat \rho$ to construct a universal model for the pullback of $\bar \rho(\bar \calA_m^+)$ under the projection $\pi_2$ of $SM\times_M SM$ to the second factor.

\begin{lemma}\label{lem:canonical maps} 
There exists a canonical map $\sigma$ such that the following diagram commutes:
{\begin{equation}\label{cd1}
\begin{CD}
\bar \calA_m^+   & @>\bar\rho>> & \Omega^*( SM)\\
@V{\sigma}VV &  &   @V{\pi_2^*}VV \\
\Omega^*(S^m,\bar \calA_{m})^+ &@>{ \hat\rho}>>  &  \Omega^*(SM\times_M SM)
\end{CD}
\end{equation}
}
\end{lemma}

\begin{proof} 
Given $g \in O(m+1)$ we define an algebra homomorphism $s_g: \bar \calA_m \to \Omega^*(S^m, \bar \calA_m)$ by taking $s_g(\theta_i) = \tilde L_g(\theta_i), s_g(\Omega_{ij}) = \tilde L_g(\Omega_{ij})$ and
  \begin{equation}
 \label{eq:formal change} 
s_g(\omega_{i0})    =  \tilde L_g(\omega_{i0}) + g_{ji}{du_j},
 \end{equation}
 where we recall the definition of $\tilde L_g$ from the discussion around  \eqref{O_action}, \eqref{eq:tilde L}.
One sees readily that if $h \in O(m)$ then
 $$
 s_{gh} = s_g \circ L_h, \quad s_{hg} = \Lambda_h \circ s_g.
 $$
 From the first relation it follows that if $\phi \in \bar \calA_m^+$ then $s_g(\phi) $ depends only on  $g\cdot e_0 =u$ and $\det g$. Thus we may define $\sigma: \bar \calA_m^+ \to \Omega^*(S^m, \bar \calA_m)$ by $\restrict{\sigma(\phi)} u := \restrict{s_g(\phi)} u$, where $g\cdot e_0 = u$ and $\det g = 1$. The second relation now implies that 
 $\sigma(\bar \calA_m^+) \subset \Omega^*(S^m, \bar \calA_m)^+$.
 
   It remains to show that $\hat \rho, \sigma$ fit into \eqref{cd1}.  Selecting a point $\xi_0$, a positive frame $b$ at $\xi_0$, and $g \in SO(m+1)$, we accomplish this by proving the corresponding fact for the 
   contingently defined maps $ \hat \rho_b,s_{ g}$ above. This amounts to the following assertion.
Let $u:= g\cdot e_0 \in S^m$, $\eta_0 := \psi_b\inv(u) \in S_{x_0}M$, and $\beta:= R_g b$. Thus $\beta$ is a positive adapted frame at $\eta_0$.
   The proof of the Lemma will  be completed by showing that the diagram
 \begin{equation}\label{cd2}
\begin{CD}
\bar \calA_m   & @>\restrict{\pi_2^*\circ\bar\rho_{  \beta}}{\xi_0,\eta_0}>> &\bigwedge^*T_{\xi_0,\eta_0}( SM \times_M SM)\\
@V{\restrict{s_{  g}}u}VV &  &   @V{r_{\xi_0}}VV \\
\bigwedge^*T_u S^m \otimes \bar \calA_{m} &@>{ \restrict{\hat\rho_b}{u}}>>  &  
\bigwedge^*T_{\eta_0} S_{x_0}M\otimes \bigwedge^*T_{\xi_0} SM
\end{CD}
\end{equation}
commutes, where the map on the right is defined as usual via the connection.

Observe first that if $\bar g:S^m \supset U \to SO(m+1)$ is a local section extending $g$, i.e.
$$
\bar g(v)\cdot e_0 = (\bar g_{00}(v), \bar g_{10}(v),\dots,\bar g_{m0}(v)) = v, \quad \bar g (u ) = g,
$$
then  \eqref{eq:formal change} may  be rewritten as
\begin{equation}\label{eq:virtual relation}
\restrict{s_{g}(\omega_{i0})   }u  = { \bar g_{00}(u) \bar g_{ji}(u)\omega_{j0} + \bar g_{0i}(u) \bar g_{ j 0}(u)  \omega_{0j} } + \bar g_{ji}(u)\restrict{d\bar g_{j0}} u.
\end{equation}
We construct a convenient choice of such an extension as follows.
 By continuity we may assume that $\xi_0\not \perp \eta_0$, or equivalently that $u \not \perp e_0$. Thus by Lemma \ref{lem:rot}, we may take 
\begin{equation}\label{eq:adapted rotation}
\bar g(\barrot_w(u)) := \barrot_w \cdot g
\end{equation}
for $w$ in some small neighborhood of $e_0$.

  Let $\bar b$ be the infinitesimally parallel extension of $b$, and define an extension $\bar \beta$ of $\beta$ to an adapted moving frame on a neighborhood of $\eta_0$ in $SM$ as follows: for $v$ in a neighborhood of $u$ in $S_{x_0}M$, take $\bar \beta(v):= R_{\bar g\psi_b(v)} b$; then extend by parallel translation along geodesics from $x_0$.
%
For $(\xi,\eta)\in SM \times_M SM$ lying in the appropriate domain we define $\barbar g(\xi,\eta)\in SO(m+1)$ by the condition $R_{\barbar g (\xi,\eta)} \bar b(\xi) = \bar \beta(\eta)$. Thus, by Lemma \ref{lem:rot} and the definition of the infinitesimally parallel extension $\bar b$,
 \begin{equation}\label{eq:pre-rotation}
 \barbar g(\xi, \eta) = \barrot_{\psi_b(\xi)}\inv \cdot (\bar g\circ \psi_b)(\eta)
 \end{equation}
 for $\xi \in S_{x_0} \setminus \{-\xi_0\}, \eta \in \psi_b\inv(U)$, so that by Lemma \ref{lem:rot} and  \eqref{eq:adapted rotation}
\begin{equation}\label{eq:constant rot}
 \barbar g(\xi, \rot^{\xi_0}_\xi(\eta_0)) = \barrot_{\psi_b(\xi)}\inv \cdot (\bar g\circ \psi_b(\rot^{\xi_0}_\xi(\eta_0)))
  = \barrot_{\psi_b(\xi)}\inv \cdot (\bar g(\barrot_{\psi_b(\xi)}(\psi_b(\eta_0))) \equiv g.
 \end{equation}
 Since the group action commutes with parallel translation, the map $\barbar g$ is determined by these restricted values: if $\xi',\eta' \in S_xM$ are the parallel translates of $\xi,\eta$ along the geodesic from $x_0$ to $x$, then $\barbar g(\xi',\eta') = \barbar g(\xi,\eta) $.

Let  $R:FM \times O(m+1) \to FM$ be the right action.
We have the locally defined commutative diagram
 \begin{equation}\label{big diagram}
\begin{tikzcd}
 \calA_m\arrow{r}{\rho}
&{\Omega^*(FM)}\arrow{r}{ R^*}   \arrow{d}{\bar \beta^*}
& \Omega^*(FM \times O(m+1))\arrow{r}{\ev_{b,g}}\arrow{d}{(\bar b, \barbar g)^*}
&
\bigwedge{}^*T_{ b}FM \otimes \bigwedge^* T_g{}O(m+1)
\arrow{d}{{\restrict{ (\bar b, \barbar g)^*}{\xi_0,\eta_0}}} \\
\bar \calA_m \arrow{r}{\bar \rho_{\bar \beta}}\arrow[hook]{u}
& \Omega^*(SM)\arrow{r}{\pi_2^*}
& \Omega^*(SM \times_M SM) \arrow{r}{\ev_{\xi_0,\eta_0}}
& \bigwedge{}^* T_{\xi_0} SM  \otimes \bigwedge{}^* T_{\eta_0} S_{x_0}M
\end{tikzcd}
\end {equation}
where, as in \eqref {cd2}, the space on  the lower right is identified with $\bigwedge{}^*T_{\xi_0,\eta_0}(SM \times_M SM)$ via the connection.
We claim that the composition of maps from the lower left corner across the top line to the lower right coincides with $\restrict{\hat \rho_b \circ s_{ g}}u$, which will imply commutativity of \eqref{cd2}. 

Note first that the derivatives $\bar b_*, \barbar g_*$ at $(\xi_0,\eta_0)$ vanish on $T_{\eta_0} S_{x_0}M ,  T_{\xi_0} SM $ respectively. The first of these statements is obvious, while the second follows from \eqref{eq:constant rot} and the definition of the connection.
By the definitions of $\bar b$ and of the realization 
maps we have
\begin{align*}
\restrict{\bar b^* (\rho(\theta_i)) }{\xi_0} = \bar \rho_b(\theta_i), & \quad
\restrict{\bar b^* (\rho(\Omega_{ij})) }{\xi_0} = \bar \rho_b(\Omega_{ij}), \\
\restrict{\bar b^* (\rho(\omega_{i0}) )}{\xi_0}& = \bar \rho_b (\omega_{i0}),\\
\restrict{\bar b^* (\rho(\omega_{ij}) )}{\xi_0}& = 0, \quad 0\notin \{i,j\} .
\end{align*}
Meanwhile, the restrictions of $\barbar g_*$ and $(\bar g\circ \psi_b)_*$  to $T^*_{\eta_0} S_{x_0}M$ clearly coincide.
Thus the basic theory of connections on principal bundles (\cite{bishop_crittenden64}, Lemma 1 of Section 5.1, and Lemma 
1 of Section 6.1.1) gives the values at $(b, g) \in FM \times O(m+1)$ as
 \begin{align*}
\restrict{ R^* \rho(\theta_i)}{b,g}  & = g_{ji}\restrict{\rho(\theta_{j})}b\\ 
\restrict{R^* \rho(\Omega_{ij})}{b,g} & = g_{k i}g_{l  j} \restrict{\rho(\Omega_{kl})}b \\
 \restrict{R^* \rho(\omega_{i0})}{b,g}    &= g_{k i}g_{l 0} \restrict{\rho( \omega_{kl}) }b+ g_{ji} {dg_{j0}} .
 \end{align*}
Applying the adjoints of the derivatives described above and comparing with the definition of $s_g$ in \eqref{eq:formal 
change} and \eqref{eq:virtual relation} completes the proof. 
\end{proof}

\begin{proof}[Proof of Proposition \ref{prop:formal forms}]
Recalling \eqref{eq:a rho}, we construct $\varphi_1$ by feeding $L_{(-1)} d\lambda_1 \otimes \phi_{kp}$ into the tensor square of \eqref{cd1}, then applying \eqref{cd0}. Note that, in applying Theorem \ref{thm:def mult} in our present 
circumstances, we have $\calF_\mu =0$, and we may ignore the $\pi^*\tau $ term.
\end{proof}

\subsection{Proof of Lemma \ref{lem:exist iso}} \label{sect:exist iso}
The proof is similar to that of Lemma \ref{lem_module_existence}, with the Alesker-Bernig formula replaced by the more 
elementary construction giving the restriction of smooth curvature measures on the level of differential forms, which we now describe. 

\subsubsection{Restrictions of curvature measures} Let $e:M^{m+1} \hookrightarrow N^{n+1}$ be an isometric immersion, $m < n$. Let $S^\perp_M N \subset \restrict{SN}M $ denote the bundle of spheres normal to $M$, with projection $q:S^\perp_MN \to M$. Consider the open submanifold $\restrictplus{SN}M := \restrict {SN}M\setminus S^\perp_MN$. Put
$$
i:\restrictplus{SN}M \to  {SN}, \quad j :S^\perp_M N \to SN.
 $$
for the inclusion maps. Normalized projection to the tangent spaces of $M$ gives a projection $p:\restrictplus{SN}M \to SM$, whose fibers are open hemispheres of dimension $n-m$.

\begin{proposition}\label{prop:iota*}
 Let $\omega \in \Omega^{n}(SN), \gamma\in \Omega^{n+1}(N)$. 
  Then the fibers of the bundles $p,q$ may be canonically oriented so that
  \begin{equation} \label{eq restriction fiber integral} e^* [ \omega,\gamma]=e^* [ \omega,0]  = [  p_*i^*\omega,  q_*j^* \omega ]\end{equation}
\end{proposition}

Although the fibers of $p$ are not compact, they may be easily compactified by a blowing up process, so that the fiber integral operator $p_*$ that occurs here is well defined.

\begin{proof}
Let a smooth polyhedron $A \subset M$ be given. The proof of Theorem 4.5 of \cite {fu 94} shows how to construct the normal cycle  $\nc_N(A)$ of $A$, considered as a subset of $N$, from its normal cycle $\nc_M(A)$ as a subset of $M$:
under a suitable choice of orientations of the fibers $S^{n-m}_+$ of $p$, we have
$$
\nc_N(A) = i_*(\nc_M(A) \times_p S^{n-m}_+) + \restrict {\nc_N(M) } A. 
$$
The conclusion now follows  from the fiber integration definition of products of currents.

Actually this construction is valid without regard to the Riemannian structure via passage to the conormal cycle. From that perspective, that the fibers are oriented canonically follows from the fact that the pair $(M,N)$ is locally diffeomorphic to the model case $(\RR^{m+1},\RR^{n+1})$.
\end{proof}
The bundle $p$ admits a natural connection, as follows. Take the subspace $SM \subset \restrictplus {SN}M$ to be parallel. Using the map $(\xi,\eta,\theta) \mapsto \cos \theta \, \xi + \sin \theta \, \eta$, the deleted bundle $\restrictplus {SN}M \setminus SM \to SM $ may be identified with the pullback to $SM$ of $S^\perp_MN \times (0,\frac \pi 2) \to M$ . The normal connection on $S^\perp_M N\to M$ coming from the Levi-Civita connection on $N$ induces a connection on the latter space as a bundle over $M$. Now pull this connection back to $SM$ via the projection to $M$. These two descriptions clearly fit together to give the desired connection on $p$. This connection now yields for each $\xi \in SM$ a canonical restriction map $r_\xi:\Omega^*(\restrictplus {SN}M )\to \Omega^*(S_\xi^+N, \bigwedge{}^* T_\xi SM)$, such that any $\phi$ is determined uniquely by the family of $r_\xi(\phi), \ \xi \in SM$.


\subsubsection{The formal model}
With Proposition \ref{prop_inv_elements2}, the following will imply Lemma \ref{lem:exist iso}.

\begin{proposition}\label{prop:iso invariants} Let $k,p,m,n$ be given. Then there exist elements $\psi_1\in \barbar \calA^+_{m,n}, \psi_2 \in  \calB^+_{m,n}$ such that for any  smooth isometric immersion $e:M^{m+1} \hookrightarrow N^{n+1}$
\begin{equation}\label{eq:I phi}
p_*i^* \bar\rho_N(\phi_{kp})= \bar\rho_{rel}(\psi_1), \quad \pi_* j^*\bar \rho_N( \phi_{kp})= \bar \rho_{rel}(\psi_2).
\end{equation}
\end{proposition}

The proof of the first assertion is  parallel to the proof of the first assertion of Proposition \ref{prop:formal forms}. 
%
Consider the sphere $S^{n-m} \subset \RR \oplus \RR^{n-m}$,  with  standard basis $e_0,e_{m+1},\dots,e_{n}$ and associated coordinates $u_0,u_{m+1},\dots,u_{n}$. The group $O(n-m)$ acts on $S^{n-m}_+$ by fixing the $0$ coordinate. Put $S^{n-m}_+$ for the open hemisphere $\{u_0 >0\}$. The group $O(m) \times O(n-m)$ admits a left action $\bar \Lambda$ on $\Omega^*(S^{n-m}_+, \bar \calA_{m,n})$ by
$$
\bar \Lambda_{g,h} \phi:= L_{g,h} \circ ({h\inv}^* \phi).
$$
Denote by $\Omega^*(S^{n-m}_+, \bar \calA_{m,n})^{++}$ the subspace $\{\phi:  \Lambda_{g,h} \phi =\det g \cdot \det h \cdot \phi\}$.


\begin{lemma}\label{lem:canonical maps 2} 
There exist canonical maps $\sigma_{rel}, \hat \rho_{rel}$ such that the following diagram commutes, where $J= \int_{S^{n-m}_+}$:
\begin{center}
\begin{tikzpicture}
  \matrix (m) [matrix of math nodes,row sep=3em,column sep=4em,minimum width=2em]
  {
     & \bar\calA_n^+ &  \\
     \bar\calA_{m,n}^{++} &  & \Omega^*({SN})\\
     \Omega^*(S^{n-m}_+,\bar\calA_{m,n})^{++} &   & \Omega^*(\restrictplus{SN}M )\\
      \bar\calA_{m,n}^+  &   & \Omega^*(SM)\\};
  \path[-stealth]
    (m-1-2) edge node [above] {$\bar \rho$} (m-2-3)
               	edge (m-2-1)
    (m-2-3) edge  node [left] {$i^*$} (m-3-3)
    (m-2-1) edge node [above] {$ $} (m-2-3)
            	edge node [left] {$\sigma_{rel}$}  (m-3-1)
    (m-3-1) edge node [above] {$\hat \rho_{rel} $} (m-3-3)
    		edge node [left] {$J$}  (m-4-1)
    (m-4-1) edge node [above] {$\bar \rho_{rel}$} (m-4-3)
    (m-3-3) edge node [left] {$p_*$} (m-4-3);
\end{tikzpicture}
\end{center}

%
\end{lemma}
%

\begin{proof}[Proof of Lemma \ref{lem:canonical maps 2}] 

Put $Q$ for the index set $\{0, m+1,\dots, n\}$. Let $O(n-m+1)$ act on $\RR^Q= \RR^{n-m+1}$ in the expected way, so that the action of the subgroup $O(n-m)$ is as described above, fixing  $e_0$.  
Let $\bar g:S^{n-m}_+ \to SO(m) \times SO(n-m+1)\subset O(n+1)$ be a section, and  define
$\sigma_{rel}: \bar \calA_{m,n} \to \Omega^*(S^{n-m}_+, \bar \calA_{m,n})$ by putting, for $u \in S^{n-m}_+$,
\begin{align}
 \restrict{\sigma_{rel}( \theta_i) }u& := \bar g_{j i}(u) \theta_{j}\notag\\ 
 \label{eq:formal change 2}  \restrict{  \sigma_{rel}(\Omega_{ij})}u & := \bar g_{k i}(u) \bar g_{l j} (u)\Omega_{kl}\\ 
 \restrict{\sigma_{rel}(\omega_{ij} )}u& := \bar g_{k i}(u)\bar g_{l j}(u) \omega_{kl} + \bar g_{ji}(u){du_j}\notag \\
& = \bar g_{k i}(u)\bar g_{l j}(u) \omega_{kl} + \bar g_{ji}(u){d\bar g_{j0}},\notag
\end{align}
%
where we take into account the relations \eqref{eq:bar mn relations}.
As in the remarks following the statement of Proposition \ref{prop:formal forms}, one readily sees that the restriction to $\bar \calA^{++}_{m,n}$ of the  resulting algebra homomorphism is independent of the choice of $\bar g$, and that the values of the restricted map enjoy the indicated symmetries.


Applying the constructions of Section \ref{sect:fiber} to the connection on the bundle $p$ described above, we may define 
the map $\hat \rho_{rel}$ by specifying a  map  into $\Omega^*(S^+_\xi N, \bigwedge{}^* T_\xi SM)$ 
for each $\xi \in SM$ . Let $b =(b_0,\dots,b_n)$ be a positive adapted frame at $\xi$, so that $b_0 = \xi$ and $b_1,\dots,b_m $ are tangent to $M$ at $x:= \pi(\xi)$. Put $\psi_b: S^+_xN\to S^{n-m}_+ $ by $\psi_b\inv(u_0,u_{m+1},\dots,u_n):= u_0 b_0 + u_\alpha b_\alpha$. Now put 
$\hat \rho_{b,rel}(\phi):= \bar \rho_{b,rel} \circ \psi_b^*\phi$. Restricting 
this map to $\phi\in \Omega^*(S_+^{n-m},\bar \calA_{m,n})^{++}$ yields $\hat \rho_{rel}$.

We now show that the top square commutes. Extend $b$ to an adapted moving frame $\bar b$ defined on a neighborhood of $\xi$ in $SM$, such that $(\bar b_0,\dots,\bar b_m) $ is constructed as in the proof of Lemma \ref{lem:canonical maps} above, and $\bar b_{m+1},\dots,\bar b_n$ depend only on the base point $\pi(\bar b_0)$ and all have covariant derivative zero at $\pi(\xi)$ with respect to the normal connection. Thus $\restrict{\bar\rho_{\bar b,rel} (\omega_{ij} )}\xi  = \restrict{\bar\rho_{\bar b,rel} (\omega_{\alpha \beta} )}\xi = 0$, $0\notin\{i,j\}$.

%
We may now use the section $\bar g$ above to construct an adapted moving frame on $\restrictplus{SN}M$ in the neighborhood of $S_{\xi_0}^+M$, by taking for $(u_0,u_{m+1},\dots,u_n) \in S^{n-m}_+$, $\xi \in SM$,
$$
\beta( u_0 \xi + u_\alpha b_\alpha):= R_{\bar g(u_0,u_{m+1},\dots,u_n)} \bar b(\xi).
$$
Let $\bar \beta$ be an extension to an adapted moving frame on an open subset of $SN$. Put $\tilde g(u_0 \xi + u_\alpha b_\alpha):= \bar g(u_0,u_{m+1},\dots,u_n)$.
We have the locally defined commutative diagram
 \begin{equation}\label{big diagram}
\begin{tikzcd}
 \calA_n\arrow{r}{\rho}
&{\Omega^*(FN)}\arrow{r}{ R^*}   \arrow{d}{\bar \beta^*}
& \Omega^*(F(M,N) \times O(n-m+1))\arrow{r}{ r_b}\arrow{d}{(\bar b\circ p, \tilde g)^*}
&
\Omega^*(O(n-m+1), \bigwedge{}^*T_{ b}F(M,N) )
\arrow{d}{} \\
\bar \calA_m \arrow{r}{\bar \rho_{\bar \beta}}\arrow[hook]{u}
& \Omega^*(SN)\arrow{r}{i^*}
& \Omega^*(\restrictplus{SN}M) \arrow{r}{r_{\xi_0}}
&\Omega^*  (S^+_{\xi_0}N, \bigwedge{}^* T_{\xi_0} SM )
\end{tikzcd}
\end {equation}
The commutativity of the top square in the diagram in the statement of the Lemma now follows by comparing the values \eqref{eq:formal change 2} with the results of the right action $R$, as in the proof of Lemma \ref{lem:I}.

We now examine the bottom square. Removing the covariance conditions from the spaces on the left, and replacing the 
horizontal maps by $ \hat \rho_{b,rel},  \bar \rho_{b,rel}$, commutativity follows from Lemma \ref{lem:fiber}. It remains to show that $J$ maps $\Omega^*(S^{n-m}_+,\bar \calA_{m,n})^{++}$ to $\bar \calA_{m,n}^+$.
Taking into account the change of orientation of the fiber, we have for any  $(g,h)\in  O(m) \times O(n-m)$
\begin{equation}\label{eq:OxO action}
J(\bar \Lambda_{g,h} \phi)   = (\det h) \bar \Lambda_{g,h} J(\phi).
\end{equation}
This implies that
 \begin{equation}\label{eq:OxO invariance}
 (g,h)\cdot J(\bar \phi_{kp}) = (\det g) J(\bar \phi_{kp}), \quad (g,h) \in O(m)\times O(n-m),
 \end{equation}
 which implies that $J$ maps $\Omega^*(S^{n-m}_+, \bar \calA_{m,n})^{++}$ to $\bar \calA_{m,n}^+$. 
  \end{proof}

\begin{proof}[Proof of Proposition \ref{prop:iso invariants}]
 Insert $\phi_{kp }\in \bar \calA_n^+$ into the top of the diagram in Lemma \ref{lem:canonical maps 2}  and chase it down in both directions. Applying  \eqref {eq:formal change 2} to the terms of the expansion \eqref{def_hatPhi}, it is clear that the only terms that yield $n-m$ factors $du_\beta$ are those for which $\{m+1,\dots, n\} \subset \{\pi_{k+1},\dots, \pi_n\} $. Thus these terms include no factors  $\Omega_{\alpha t}$.    Since  only these terms contribute to the fiber integral $J$, it follows that $\psi_1:=J \circ \sigma_{rel}(\phi_{kp}) \in \barbar \calA_{m,n}^+$. That the bottom square of the diagram commutes is now immediate.

We omit the simpler proof of the second assertion.
\end{proof}

\section{An application to hermitian integral geometry}\label{sect:hermitian} 

We apply our results to study the space $\calC^n_\lambda$ of invariant curvature measures, and the corresponding algebra $\calV^n_\lambda$, on the complex space forms $\CC P^n_\lambda$.

\subsection{Review of hermitian integral geometry} Let us recall the main notions of this subject, developed mostly in  \cite{bfs}.

\subsubsection{General theory} A {\it Riemannian isotropic space} is a pair $(M,G)$, where $M$ is a Riemannian manifold and $G$ is a group of isometries of $M$ that acts transitively on the tangent sphere  bundle $SM$. For such a space
there exist kinematic formulas both for $G$-invariant curvature measures and for $G$-invariant valuations, the latter being the image of the former under the globalization map:
$$
\begin{CD}
\calC^G(M) & @>K>> &  \calC^G(M)\otimes \calC^G(M) \\
@V{\glob }VV &  &@VV\glob \otimes \glob V\\
\calV^G(M) & @>k>> &  \calV^G(M)\otimes \calV^G(M)
\end{CD}
$$
where the coproducts $K,k$ are cocommutative and coassociative. Picking a representative point $o \in M$ and setting $G_o\subset G$ to be the subgroup $\{g:go = o\}$, the pair $(T_oM, \overline{G_o})$ is again a Riemannian isotropic space, where $\overline{G_o}$ is the group generated by the group of translations and the derivative of the action of $G_o$.  Howard's transfer principle (\cite{bfs}, Theorem 2.23) implies that the natural map $\calC^G(M) \to \calC^{\overline {G_o}}(T_oM)$ is an isomorphism of graded coalgebras.

By the  fundamental theorem of algebraic integral geometry (\cite{bfs}, Theorem 2.21), the kinematic operator $k$ is adjoint to the multiplication of $\calV^G(M)$ 
via a natural Poincar\'e duality map $p: \calV^G (M)\to \calV^{G}(M)^*$. By the same token, the action of $\calV^G(M)$ on $\calC^G(M)$ encodes information about the operator $K$ (\cite{bfs}, Corollary 2.20).
In view of the obvious fact that the Riemannian curvature measures are invariant under isometries, we now observe:
\begin{lemma} \label{lem:image R} If $(M,G)$ is a Riemannian isotropic space then $\calR(M) \subset \calC^G(M)$ and $\calLK(M) \subset \calV^G(M)$.
\end{lemma}
It follows that  the realization at $M$ of the action of $\RR[t]$ on $\calR$ carries kinematic information. We now illustrate this principle by examining the case of the integral geometry of complex space forms.

\subsubsection{The case of the complex space forms}\label{sect:c space forms} Put  $\CC P^n_\lambda$ for the complex space form of complex dimension $n$ and holomorphic sectional curvature $4\lambda$, equipped with its  isometry groups $G^n_\lambda$. For $\lambda = 0$, we take $G^n_0:= \overline{U(n)}$, and the resulting isotropic euclidean space is canonically isomorphic to each tangent space $(T_o\CC P^n_\lambda, (G^n_\lambda)_o)$ described in the general case above. This 
 implies that for each $n$ there exists a single model $\calC^n$ canonically isomorphic as a coalgebra to each $ \calC^{G^n_\lambda}(\CC P^n_\lambda), \ \lambda \in \RR$. 

By definition, an element of $\calC^{G^n_\lambda}(\CC P^n_\lambda)$ is angular iff the corresponding element of $\calC^n$ is.  The space of all such elements is graded, and denoted by $\ang^n = \bigoplus_{k=0}^{2n}\ang^n_k$. 
The space $\calC^n$ decomposes as
$$
\calC^n = \ang^n \oplus \Null^n_\lambda
$$
where  $\Null^n_\lambda := \ker (\glob^n_\lambda)$ (cf. \cite{bfs}).
The space $\ang^n$ admits a canonical graded basis
$$
\Delta_{kq} \in \ang^n_k, \quad \max(0, k-n ) \le q \le \frac k 2 \le n,
$$
characterized by the property that $\Delta_{kq}(\CC^p \oplus \RR^{l-2p}, \cdot)= \delta^{k,q}_{l,p} 
\restrict{\mathcal  {H}^k}{\CC^q \oplus \RR^{k-2q}}$, where $\RR^{l-2p}$ denotes a subspace of $\CC^n$ of the 
given real dimension that is orthogonal both to $\CC^p$ and to $\sqrt{-1}$ times itself.

By Lemma 3.6  of \cite {bfs} and the remarks preceding it, there are natural restriction maps $r_n:\calC^{n+1} \to \calC^n$ that act formally as the identity on the span of the $\Delta_{kp}, k\le n$, and that respect the decompositions above.
Take $\calC^\infty= \ang^\infty \oplus \Null_\lambda^\infty$ to be the inverse limit of this system. The inclusions also induce surjective algebra homomorphisms $\calV^{n+1}_\lambda\to \calV^n_\lambda$, intertwining the actions on $\calC^{n+1}, \calC^n$ and the restriction maps $r_n$. The inverse limit algebra $\calV^\infty_\lambda$ thus acts on $\calC^\infty$, and includes a canonical copy of $\RR[[t]]$ obtained as the inverse limit of the $\calLK(\CC P^n_\lambda)$. The generator thus gives rise to an operator on $\calC^\infty$ that we denote by $t_\lambda$, partially described in Theorem 6.7 of \cite{bfs}.  Clearly $t_\lambda(\Null^\infty_\lambda)\subset \Null^\infty_\lambda$. 

 By Proposition \ref{prop:C is angular} and  Lemma \ref{lem:image R},
the realization maps $\calR \to \calC(\CC P^n_\lambda)$ may be represented as maps
\begin{equation}\label{eq:image rho} 
\notag\rho^n_\lambda:\calR \to \ang^n\subset  \calC^n.
\end{equation}
Applying Corollary \ref{cor:totally geodesic} to the standard totally geodesic inclusions $\CC P^n_\lambda \hookrightarrow \CC P^{n+1}_\lambda $, we deduce that
\begin{equation}\label{eq:restrictions commute}
r_n \circ \rho^{n+1}_\lambda  = \rho^n_\lambda .
\end{equation}
Thus there are maps $\rho_\lambda:\calR\to \ang^\infty, r_n:\calC^\infty \to\calC^n$, such that $\rho_\lambda^n = r_n\circ \rho_\lambda$.  

In the next section we show first that if $\lambda \ne 0$ then $\rho_\lambda$ is a linear isomorphism. Therefore $t_\lambda $ stabilizes $\ang^\infty$, by Theorem \ref{thm:LK}. 
We then give explicitly the relation between the  $\rho_\lambda(C_{kp}), \Delta_{lq}$, and express the action of $t_\lambda$ in terms of the latter.

\subsection{The correspondence $\calR \leftrightarrow \calC^\infty$}
\subsubsection{The exponential generating functions for $\widetilde{\calLK}(\CC P^n_\lambda)$}
We abbreviate 
$$C^\lambda_{kp}:= \rho_\lambda(C_{kp}), \quad \bar \Lambda^\lambda_k := \rho_\lambda(\bar\Lambda_k).
$$

We 
introduce the alternative graded basis
\begin{equation}\label{eq:delta tilde} 
\widetilde \Delta_{kl}:= \sum_{j\ge 0} \binom {l+j} l \Delta_{k,l+j} 
\end{equation}
so that
\begin{equation}\label{eq:delta tilde inv} 
 \Delta_{kj}:= \sum_{l\ge 0} (-1)^j\binom {l+j} l  \widetilde \Delta_{k,l+j} .
\end{equation}
Thus, in the case of complex euclidean space (viz. $\lambda = 0$), the globalizations of the $\tilde \Delta$ coincide with the Tasaki valuations introduced in \cite{hig}.

We may  express the Lipschitz-Killing curvature measures of the $\CC P^n_\lambda$ in terms of the $\tilde \Delta_{kp}$ by means of exponential generating functions. Define 
\begin{align}
g_k(z,y)&:= \binom{2k}k z^k (1-4z)^{-k-\frac 1 2} (1-4y)^{-\frac 3 2}\\
h_k(z,y)&:=  z^k (1-4z)^{-k-\frac 3 2} (1-4y)^{-\frac 3 2}
\end{align}

\begin{proposition}\label{prop:gen fns}
Define the graded $\omega$-linear maps $L,M:\RR[[z,y]]\to \ang^\infty$ by
\begin{align*}
L(z^my^p)&:= m!p! \tilde \Delta_{2m+2p,p}, \\
M(z^my^p) &:= m!p! \tilde \Delta_{2m+2p+1,p}
\end{align*}
of degrees $0,1$ respectively. Then for any $\lambda \in \RR$
\begin{align}
\label{eq:Lambda even}\left(\frac {4}\lambda\right)^{k}L\left( g_k\left(\frac{\lambda z}{4\pi},\frac{\lambda y}{4\pi} \right)\right) & = \bar\Lambda^\lambda_{2k} ,\\
\label{eq:Lambda odd}\frac 2 \pi \left(\frac {16}\lambda\right)^{k} M\left(h_k\left(\frac{\lambda z}{4\pi},\frac{\lambda y}{4\pi} \right)\right)
& = \bar \Lambda_{2k+1}^\lambda.
\end{align}
\end {proposition}
In other words, the $\bar \Lambda^\lambda_j$ may be expressed as sums of the  $\tilde \Delta_{kp}$, with  coefficients given by the values at $(0,0)$ of appropriate partial derivatives of the functions on the left of \eqref{eq:Lambda even}, \eqref{eq:Lambda odd}.
\begin{proof}
This is a straightforward modification of Lemma 3.12 of  \cite{bfs}.
\end{proof}

The following now implies Proposition \ref{prop:hermitian lemma} (1).
\begin{corollary} \label{cor:surjective} 
Suppose $\lambda \ne 0$. Then the map $\rho_\lambda:\calR\to \ang^\infty$ is surjective. Furthermore, if $ m\le  n $, then $\restrict{\rho^n_\lambda} {\calR_m}$ is an isomorphism onto $\ang^n_m$.
\end{corollary}
\begin{proof} 
By \eqref{eq:restrictions commute}, it is enough to prove the second assertion. Suppose $2l:=m\le n$. Clearly $\tilde \Delta_{2l,0},\dots,\tilde \Delta_{2l,l}$ are linearly independent. On the other hand, for $j\le l$ the component of degree $2l$ in $\bar \Lambda_{2j}$ is $C_{2l, j-l}$, up to a nonzero factor. Thus the matrix expressing the $C^\lambda_{2l,\cdot}$ in terms of the $\tilde \Delta_{2l, \cdot}$ is given by the values at $(0,0)$ of $\partial_z^m \partial_y^p  L'$ with $m+p = l, k\le l$, where $L'$ is the function on the left of \eqref{eq:Lambda even}. The entries vanish if $m< k$, and are nonzero if $m=k$. Thus the matrix is  the reflection of a triangular matrix with nonzero diagonal entries. The case where $m$ is odd is similar.
\end{proof}

\subsubsection{Expressing the $\tilde \Delta_{kp}$ in terms of the $C^\lambda_{lq}$} Fix $\lambda \ne 0$.
Recalling the identification $\alpha:\calR\leftrightarrow \RR[[\xi,\eta]] $, we introduce new coordinates of degree 2 by
$$
\bar \xi :=  {\xi^2} , \quad \bar \eta := \frac {\eta} { \lambda} - \bar \xi .
$$ 
Take $z,y$ both have degree 2, and define the graded $\omega$-linear maps $\calL,\calM: \RR[[z,y]]\to \RR[[\xi, \eta]]$  by 
\begin{align*}
\calL\left(z^my^p\right)&:= \pi^{m+p}\frac{\binom{m+p} p}{ (2p+1)\binom {2p} p \binom{2m} m} \bar\xi^m \bar \eta^p,\\
\calM\left(z^my^p\right)&:= \left(\frac \pi 4\right)^{m+p}\frac{(2m+2p+1)\binom{2m+2p}{m+p}\binom {m+p} p}{(2m+1)(2p+1) \binom{2m} m \binom {2p} p} \bar\xi^m \bar \eta^p.
\end{align*}
Clearly $\calL, \calM$ intertwine the operators $z^k \partial_z^k, \bar\xi^k \partial^k_{\bar\xi}$, and by \eqref {eq:maclaurin1/2}, \eqref{eq:maclaurin3/2},
\begin{align*}
\calL\left(g_0(z,y)\right)&= \left(1- \pi({\bar\xi+\bar \eta})\right)\inv, \\
\calM\left(h_0(z,y)\right)&= \left(1- \pi({\bar\xi+\bar \eta})\right)^{-\frac 3 2}.
\end{align*}
Since $z^k \partial_z^k g_0= k!g_k$ and $ z^k \partial_z^k h_0 =\frac{(2k+1)!}{k!}h_k$, we  thus compute 
\begin{align*}
\calL(g_k(z,y)) &= \pi^k\bar\xi^k \left(1-\pi(\bar\xi+\bar \eta)\right)^{-k-1} ,\\
\calM(h_k(z,y)) &= \left(\frac \pi 4\right)^k\bar\xi^k \left(1-\pi(\bar\xi+\bar \eta)\right)^{-k-\frac 3 2}, \\
\end{align*}
and
\begin{align*}
\calL\left(\left(\frac {4}\lambda\right)^{k} g_k\left(\frac{\lambda z}{4\pi},\frac{\lambda y}{4\pi} \right) \right)&=
\xi^{2k} \left(1-\frac \eta 4 \right)^{-k-1}  = \alpha(\bar\Lambda_{2k} ),\\
\xi \cdot\calM\left(\left(\frac {16}\lambda\right)^{k} h_k\left(\frac{\lambda z}{4\pi},\frac{\lambda y}{4\pi} \right)\right)&= 
\xi^{2k+1}\left(1-\frac \eta 4\right)^{-k-\frac 3 2}= \alpha(\bar \Lambda_{2k+1}).
\end{align*}

Proposition \ref{prop:gen fns} now
implies 
 that  $f:=\rho_\lambda \circ\alpha\inv \circ \calL\circ L\inv$ 
maps each $\bar \Lambda_{2k}^\lambda$ to itself, and that $g:=\rho_\lambda\circ \alpha\inv \circ m_{\frac {\pi \xi}2}\circ \calM \circ {M}\inv $  maps each $\bar \Lambda_{2k+1}^\lambda$ to itself, where $m_{\frac {\pi \xi}2}$ denotes multiplication by ${\frac {\pi \xi}2}$. As graded maps, $f,g$ act also as the identity on the  graded components of these elements.
But by the proof of Corollary \ref{cor:surjective}, these graded components  constitute an  $\omega$-basis for $\ang^\infty_{even}$ and $\ang^\infty_{odd}$, respectively. Therefore $f,g$  are the identity maps on these spaces. 
In other words,
putting 
$$
c_{mp}:=\frac{(m+p)!}{(2m)! (2p+1)! }, \quad d_{mp}:=\frac{(2m+2p+1)! }{(m+p)!(2m+1)! (2p+1)! },
$$
we obtain
\begin{align}
\notag\tilde \Delta_{2m+2p,p} &=\pi^{m+p}c_{mp}\rho_\lambda \circ\alpha\inv(\bar\xi^{m} \bar \eta^p)\notag\\
\label{eq:bar form 1}&= \pi^{m+p}c_{mp}\rho_\lambda \circ\alpha\inv\left(\xi^{2m}  \left(\frac {\eta}{\lambda} - {\xi^2} \right)^p\right)\\
\notag &= \pi^{m+p}\lambda^{-p}c_{mp}\sum_{j=0}^p \left(- { \lambda}\right)^j \binom p j 
C_{2m+2p,p-j}^\lambda.
\end{align}
and similarly
\begin{align}
\label{eq:bar form2}\tilde \Delta_{2m+2p+1,p} 
&= \frac  \pi 2 \left(\frac \pi 4\right)^{m+p}d_{mp}\rho_\lambda \circ \alpha \inv\left(\xi \bar\xi^{m} \bar \eta^p\right)\\
\notag &= \frac \pi{2\lambda^p} \left(\frac \pi 4\right)^{m+p} d_{mp} \sum_{j=0}^p (-\lambda)^j \binom p j  C^\lambda_{2m+2p+1,p-j}
\end{align}

The relations \eqref{eq:bar form 1}, \eqref{eq:bar form2} admit the common form
 \begin{equation}
 \tilde \Delta_{kp}    = \left( \frac{\pi}{2}\right)^{\lceil\frac k 2 \rceil}  \frac{k!!}{(k-2p)!(2p+1)!}\sum_{j=0}^p   (-1)^{p-j} \lambda^{-j}  \binom{p}{j}C^\lambda_{kj}.
 \end{equation}
 This expression is easily inverted:
\begin{equation}
C^\lambda_{kj}= \left( \frac{2}{\pi}\right)^{\lceil\frac k 2 \rceil} \frac{\lambda^j}{k!!}\sum_{p=0}^j \binom{j}{p}(k-2p)!(2p+1)!\tilde \Delta_{kp}.
\end{equation}

\subsubsection{Action of $t_\lambda$} We may now use \eqref{eq:bar form 1} and Theorem \ref{thm:LK} to compute
\begin{equation*}
t_\lambda\cdot \tilde \Delta_{2m+2p,p}  =  \pi^{m+p}c_{mp}\rho_\lambda \circ\alpha\inv\left(\xi\left(1-\frac \eta 4\right)^{-\frac 1 2} \bar \xi^m \bar \eta^p\right)
\end{equation*}
where
\begin{align*}
\left(1-\frac \eta 4\right)^{-\frac 1 2} \bar \xi^m \bar \eta^p&=\left(1-\frac \lambda 4\left(\bar\eta +\bar \xi\right)\right)^{-\frac 1 2}
\bar \xi^m \bar \eta^p\\
&= \sum_{n\ge m,q\ge p}\left(\frac \lambda {16}\right)^{n+q-m-p}\binom{2n+2q-2m-2p}{n+q-m-p} \binom{n+q-m-p}{n-m} \bar\xi^{n}\bar\eta^q
\end{align*}
so that, by  \eqref{eq:bar form2}, 
\begin{align}\label{eq:t even}
t_\lambda&\cdot \tilde \Delta_{2m+2p,p} = \\
\notag &\frac {2 c_{mp}} \pi  \left(\frac{16\pi}{\lambda}\right)^{m+p} \sum_{n\ge m,q\ge p} d_{nq}\inv\left(\frac \lambda {4\pi}\right)^{n+q}\binom{2n+2q-2m-2p}{n+q-m-p} \binom{n+q-m-p}{n-m} 
\tilde \Delta_{2n+2q+1,q}
\end{align}
Similarly,
\begin{align}\label{eq:t odd}
t_\lambda&\cdot \tilde \Delta_{2m+2p+1,p}  = \\
\notag & \frac  {d_{mp}} 2 \left(\frac {4\pi}\lambda\right)^{m+p}\sum_{n\ge m,q\ge p} c_{n+1,q}\inv\left(\frac \lambda {16\pi}\right)^{n+q}\binom{2n+2q-2m-2p}{n+q-m-p} \binom{n+q-m-p}{n-m} 
\tilde \Delta_{2n+2q+2,q}
\end{align}

\bigskip

\begin{theorem}\label{thm:t action} $t_\lambda\cdot \tilde \Delta_{kp} =$ 
$$\frac{2\omega_{k-1}}{\omega_{k}} \frac{(k-1)!!}{(k-2p)!(2p+1)!} \sum_{l\geq0, q\geq p} \frac{(k+2l-2q+1)!(2q+1)!}{(k+2l+1)!!} \left( \frac{\lambda}{8\pi}\right)^l  \binom{2l}{l} \binom{l}{q-p} \tilde \Delta_{k+2l+1,q}$$
\end{theorem}
\begin{proof}
Substituting $n+q=m+p+l$  and using the fact that 
 $$\frac{2\omega_{k-1}(k-1)!!}{\omega_{k}k!!}  = \begin{cases} \frac{2}{\pi} & \text{if } k \text{ is even}\\
                                                          1 & \text{if } k \text{ is odd}\\
                                                          \end{cases}$$
this expression for $t\cdot \tilde \Delta_{kp} $  follows immediately from \eqref{eq:t even}, \eqref{eq:t odd}.

\end{proof}

\end{document}